\documentclass[twoside,a4paper]{article}
\usepackage{amsmath,amsthm,amssymb,mathrsfs}
\usepackage{enumerate}
\usepackage[dvipdfmx]{pict2e}
\usepackage{bbm,cases}
\usepackage{color}
\usepackage{graphicx}

\usepackage{floatflt}

\usepackage{mathptmx}
\usepackage[scaled]{helvet}

\usepackage{float}

\newtheorem{dfn}{Definition}[section]
\newtheorem{thm}[dfn]{Theorem}
\newtheorem{lem}[dfn]{Lemma}
\newtheorem{cor}[dfn]{Corollary}
\newtheorem{rem}[dfn]{Remark}

\newtheorem{prop}[dfn]{Proposition}\makeatletter

             \@addtoreset{equation}{section}
\makeatother

\newcommand{\x}{{\mathbbm x}}

\newcommand{\dis}{\displaystyle}

\newcommand{\Z}{\mathcal{Z}}

\newcommand{\W}{\mathcal{W}}
\renewcommand{\i}{\mathbbm{i}}

\allowdisplaybreaks[4]
\begin{document}
\title{Free energy of directed polymers in random environment in $1+1$-dimension at high temperature}
\author{Makoto Nakashima \footnote{nakamako@math.nagoya-u.ac.jp, Graduate School of Mathematics, Nagoya University, Furocho, Chikusaku, Nagoya, Japan } }
\date{}
\pagestyle{myheadings}
\markboth{Makoto Nakashima}{On the universality of $1+1$ DPRE}
\maketitle

\sloppy

\begin{abstract}
We consider the free energy $F(\beta)$ of the directed polymers in random environment in $1+1$-dimension. It is known that $F(\beta)$ is of order $-\beta^4$ as $\beta\to 0$ \cite{AleYil,Lac,Wat}.  In this paper, we will prove  that under a certain condition of the potential, \begin{align*}
\lim_{\beta\to 0}\frac{F(\beta)}{\beta^4}=\lim_{T\to\infty}\frac{1}{T}P_\Z\left[\log \Z_{\sqrt{2}}(T)\right] =-\frac{1}{6},
\end{align*}
where $\{\mathcal{Z}_\beta(t,x):t\geq 0,x\in\mathbb{R}\}$ is the unique mild solution to the stochastic heat equation \begin{align*}
\frac{\partial}{\partial t}\Z=\frac{1}{2}\Delta \mathcal{Z}+\beta \mathcal{Z}{\dot{\mathcal W}},\ \ \lim_{t\to 0}\mathcal{Z}(t,x)dx=\delta_{0}(dx),
\end{align*}
where $\mathcal{W}$ is a time-space white noise 
and \begin{align*}
\mathcal{Z}_\beta(t)=\int_\mathbb{R}\mathcal{Z}_\beta(t,x)dx.
\end{align*}

\end{abstract}


\vspace{1em}{\bf Key words:}  Directed polymers,  Free energy, Universality, Continuum directed polymer.

\vspace{1em}
We denote by  $(\Omega, {\cal F},P )$ a probability space. We denote by $P[X]$ the expectation of random variable $X$ with respect to $P$.  Let $\mathbb{N}_0=\{0,1,2,\cdots\}$, $\mathbb{N}=\{1,2,3,\cdots\}$, and $\mathbb{Z}=\{0,\pm 1,\pm 2,\cdots\}$.  
Let $C_{x_1,\cdots,x_p}$ or $C(x_1,\cdots,x_p)$ be a non-random constant which depends only on the parameters $x_1,\cdots,x_p$. 



 \section{Introduction and main result}

      Directed polymers in random environment was introduced by Henly and Huse in the physical literature to study the influence by impurity of media to  polymer chain \cite{HenHus}. In particular, random media is given as i.i.d.\,time-space random variables and the shape of polymer is achieved as time-space path of walk whose law is given by Gibbs measure with the inverse temperature $\beta\geq 0$, that is, time-space trajectory $s$  up to time $n$ appears as a realization of a polymer by the probability \begin{align*}
      \mu_{\beta,N}(s)=\frac{1}{Z_{\beta,N}}\exp\left(\beta H_N(s)\right)P_S^0(S_{[0,N]}=s), \ \ \ \ s\in \left(\mathbb{Z}^{d}\right)^{N+1},
      \end{align*}
where  $H_N(s)$ is a Hamiltonian of  the trajectory $s$, $(S,P_S^0)$ is the simple random walk on $\mathbb{Z}^d$ starting from $x\in\mathbb{Z}^d$, $S_{[0,N]}=(S_0,S_1,\cdots,S_N)\in \left(\mathbb{Z}^{d}\right)^{N+1}$, and $Z_{\beta,N}$ is the normalized constant  which is called the quenched  partition function.

There exists $\beta_1$ such that if $\beta<\beta_1$, then the effects by random environment are weak and if $\beta>\beta_1$, then environment has a meaningful influence. This phase transition is characterized by the uniform integrability of the normalized partition functions. Also, we have another phase transition characterized by the non-triviality of the free energy, i.e. there exists $\beta_2$ such that if $\beta<\beta_2$, then the free energy is trivial and if $\beta>\beta_2$, then the free energy is non-trivial.  The former phase transition is referred to weak versus strong disorder phase transition and the latter one is referred to strong versus very strong disorder phase transition. We have some  known results on the phase transitions: $\beta_1=\beta_2=0$ when $d=1,2$ \cite{ComVar,Lac} and $\beta_2\geq \beta_1>0$ when $d\geq 3$ \cite{Bol,ComShiYos}. In particular, the best lower bound of $\beta_1$ is obtained by Birkner et.al.\,by using size-biased directed polymers and random walk pinning model \cite{Bir2,BirGreden2,Nak3}.

There are a lot of progressions for $\mathbb{Z}^d$-lattice model in three decades\cite{Bol,CarHu,ComShiYos,ComShiYos2,CarHu2,ComVar,Lac,BerLac2}. Recently, the KPZ universality class conjecture for $d=1$ case has been focused and was confirmed for a certain environment \cite{Sep,GeoSep,ComNgu}. The recent progressions are reviewed in \cite{Com}.

\subsection{Model and main result}

To define the model precisely, we introduce some random variables.
\begin{itemize}
\item (Random environment) Let $\{\eta(n,x):(n,x)\in\mathbb{N}\times \mathbb{Z}^d\}$ be $\mathbb{R}$-valued i.i.d.\,random variables with $\lambda(\beta)=\log Q[\exp\left(\beta \eta(n,x)\right)]\in\mathbb{R}$ for any $\beta\in\mathbb{R}$, where $Q$ is the law of $\eta$'s.
\item (Simple random walk) Let $(S,P_S^x)$ be a simple random walk on $\mathbb{Z}^d$ starting from $x\in\mathbb{Z}^d$. We write $P_S=P_S^0$ for simplicity.
\end{itemize}

Then, the Hamiltonian $H(s)$ is given by \begin{align*}
H_N(s)=H_N(s,\eta)=\sum_{k=1}^N \eta(k,s_k),\ \ \ s=(s_0,\cdots,s_N)\in \left(\mathbb{Z}^d\right)^{N+1},
\end{align*}
and \begin{align*}
Z_{\beta,N}=Z_{\beta,N}(\eta)=P_S\left[\exp\left(\beta \sum_{k=1}^N \eta(k,S_k)\right)\right].
\end{align*}

It is clear that \begin{align*}
Q\left[Z_{\beta,N}(\eta)\right]=\exp\left(N\lambda(\beta)\right)
\end{align*}
for any $\beta\in\mathbb{R}$.

The normalized partition function is defined by \begin{align}
 W_{\beta,N}(\eta)&=\frac{Z_{\beta,N}(\eta)}{Q\left[Z_{\beta,N}(\eta)\right]}\notag\\
 &=Z_{\beta,N}(\eta)\exp\left(-N\lambda(\beta)\right)\notag\notag\\
 &=P\left[\exp\left(\beta H_N(S)-N\lambda(\beta)\right)\right]\notag\\
 &=P\left[\prod_{k=1}^N\zeta_{k,S_k}(\beta,\eta)\right],\label{partfn}
 \end{align}
 where we write for each $(n,x)\in\mathbb{N}\times \mathbb{Z}^d$\begin{align*}
 \zeta_{n,x}(\beta,\eta)=\exp\left(\beta \eta(n,x)-\lambda(\beta)\right).
 \end{align*}
 
Then, the following limit exists $Q$-a.s.\,and $L^1(Q)$ \cite{ComShiYos,ComYos}:
\begin{align}
F(\beta)&=\lim_{N\to \infty}\frac{1}{N}\log W_{\beta,N}(\eta)\notag\\
&=\lim_{N\to\infty}\frac{1}{N}Q\left[\log W_{\beta,N}(\eta)\right]\notag\\
&=\sup_{N\geq 1}\frac{1}{N}Q\left[\log W_{\beta,N}(\eta)\right].\label{free}
\end{align}
The limit $F(\beta)$ is a non-random constant and called the quenched free energy. Jensen's inequality implies that \begin{align*}
F(\beta)\leq \lim_{N\to\infty}\frac{1}{N}\log Q\left[W_{\beta,N}(\eta)\right]=0.
\end{align*}
It is known that $F(\beta)<0$ if $\beta\not=0$ when $d=1,2$ \cite{ComVar,Lac} and $F(\beta)=0$ for sufficiently small $|\beta|$ when $d\geq 3$. Recently, the asymptotics of $F(\beta)$ near high temperature ($\beta\to 0$) are studied:\begin{align*}
&F(\beta)\asymp -\beta^{4},\ \ \text{if }d=1
\intertext{\cite{Lac, Wat, AleYil} and}
&\log |F(\beta)|\sim -\frac{\pi}{\beta^2},\ \ \text{if }d=2
\end{align*}
\cite{Lac,BerLac2}.

In particular, it is conjectured that when $d=1$, \begin{align*}
\lim_{\beta\to 0}\frac{1}{\beta^4}F(\beta)=-\frac{1}{24},
\end{align*}
where $\frac{1}{24}$ appears in the literature of stochastic heat equation or KZP equation \cite{BerGia,AmiCorQua}.

Our main result answers this conjecture in some sense.
\begin{thm}\label{main}
Suppose $d=1$. We assume the following concentration inequality:

{There exist $\gamma\geq 1$, $C_1,C_2\in(0,\infty)$ such that for any $n\in\mathbb{N}$ and  for any convex and $1$-Lipschitz function $f:\mathbb{R}^n\to \mathbb{R}$,}
\begin{align}Q\left(|f(\omega_1,\cdots,\omega_n)-Q[f(\omega_1,\cdots,\omega_n)]|\geq t\right)\leq C_1\exp\left(-C_2t^{\gamma}\right),\label{A}
\end{align}
where $1$-Lipschitz means $|f(x)-f(y)|\leq |x-y|$ for any $x,y\in\mathbb{R}^n$ and $\omega_1,\cdots,\omega_n$ are i.i.d.~random variables with the marginal law $Q(\eta(n,x)\in dy)$.

 Then, we have \begin{align*}
\lim_{\beta\to 0}\frac{1}{\beta^4}F(\beta)=-\frac{1}{6}.
\end{align*}
\end{thm}

The constant $-\frac{1}{6}$ appears as the limit of the free energy of the continuum directed polymers (see Lemma \ref{freeenecdp}):
 \begin{align*}
F_\Z(\sqrt{2})=\lim_{T\to\infty}\frac{1}{T}P_\Z\left[\log \int_{\mathbb{R}}\mathcal{Z}_{\sqrt{2}}^x(T,y)dx\right]=-\frac{1}{6},
\end{align*}
where $\mathcal{Z}^x_{\beta}(t,y)$ is the unique mild solution to the stochastic heat equation \begin{align*}
\partial \mathcal{Z}=\frac{1}{2}\Delta \mathcal{Z}+\beta\mathcal{Z}\dot{\mathcal{W}},
\end{align*}
with  the initial condition $\dis \lim_{t\to 0 }\mathcal{Z}(t,y)dy=\delta_x(dy)$ and $\mathcal{W}$ is a time-space white noise and $P_\Z$ is the law of $\Z_\beta^x$.
We write \begin{align*}
\Z_{\beta}^x(t)=\int_\mathbb{R}\Z_{\beta}^x(t,y)dy
\end{align*}
and $\Z_\beta(t)=\Z_\beta^0(t)$ for simplicity. $-\frac{1}{6}$ seems to be different from the value  $-\frac{1}{24}$ in  the conjecture. However, it has the relation \begin{align*}
-\frac{1}{6}=-\frac{(\sqrt{2})^4}{24}
\end{align*}
and $\sqrt{2}$ appears from the periodicity of simple random walk. Thus, the conjecture it true essentially.

\begin{rem}
Assumption (\ref{A}) are given in \cite{CarTonTor} for pinning model. Under this assumption, $\{\eta(n,x):n\in\mathbb{N},x\in\mathbb{Z}\}$ satisfies a good concentration property (see Lemma \ref{tal}). It is known that the following distribution satisfies (\ref{A}).
\begin{enumerate}
\item If $\eta(n,x)$ is bounded, then (\ref{A}) holds for $\gamma=2$ \cite[Corollary 4.10]{Led}.
\item If the law of $\eta(n,x)$ satisfies a log-Sobolev inequality (for example Gaussian distribution), then (\ref{A}) holds with $\gamma=2$ \cite[Theorem 5.3, Corollary 5.7]{Led}
\item If the law of $\eta(n,x)$ has the probability density with $c_\gamma \exp\left(-|x|^\gamma\right)$, then (\ref{A}) holds with $\gamma\in [1,2]$ \cite[Proposition 4.18,Proposition 4.19]{Led}.
\end{enumerate}
\end{rem}

\subsection{Organization of this paper} This paper is structured as follows:
\begin{itemize}
\item We first give the strategy of the proof of our main result in section \ref{2}.
\item Section \ref{3} is devoted to prove the statements mentioned in section \ref{2} related to  discrete directed polymers. 
\item Section \ref{4} is also devoted to  prove the statement mentioned in section \ref{2} related to continuum directed polymers.
\end{itemize}


\section{Proof of Theorem \ref{main}}\label{2}

In this section, we give the proof of Theorem \ref{main}. 
\subsection{Proof of limit inferior}
The idea is simple. Alberts, Khanin and Quastel proved the following limit theorem. 

\begin{thm}\label{thmref}$($\cite{AlbKhaQua}$)$ Suppose $d=1$. Let $\{\beta_n:n\geq 1\}$ be an $\mathbb{R}$-valued sequence with $\beta_n\to 0$ and $r>0$. Then,  
the sequence $\dis \{W_{r\beta_n,\lfloor T\beta_n^{-4}\rfloor}(\eta):n\geq 1\}$ is $L^2$-bounded and converges in distribution to a random variable $\Z_{r\sqrt{2}}(T)$ for each $T>0$.
\end{thm}

Combining this with (\ref{free}), we have that \begin{align*}
\frac{1}{\lfloor T\beta_n^{-4}\rfloor}Q\left[\log W_{\beta_n,\lfloor T\beta_n^{-4}\rfloor}(\eta)\right]\leq F(\beta_n)
\end{align*}for any $n\geq 1$ and $t>0$, i.e. \begin{align}
\frac{\beta_n^{-4}}{\lfloor T\beta_n^{-4}\rfloor}Q\left[\log W_{\beta_n,\lfloor T\beta_n^{-4}\rfloor}(\eta)\right]\leq \frac{1}{\beta_n^{4}}F(\beta_n).\label{inf}
\end{align}

Thus, if $ \log W_{\beta_n,\lfloor T\beta_n^{-4}\rfloor}(\eta)$ is uniformly integrable, then we have that  \begin{align}
\frac{1}{T}P_\Z\left[\log \Z_{\sqrt{2}}(T)\right]\leq \varliminf_{n\to \infty}\frac{1}{\beta_n^{4}}F(\beta_n).\label{inf}
\end{align}

Taking the limit in $T$, we have that \begin{align}
\lim_{T\to \infty}\frac{1}{T}P_\Z\left[\log \Z_{\sqrt{2}}(T)\right]\leq  \varliminf_{n\to \infty}\frac{1}{\beta_n^{4}}F(\beta_n).\label{limcdp}
\end{align}

Therefore, it is enough to show the following lemmas.
\begin{lem}\label{lpbdd}
Suppose $d=1$. We assume (\ref{A}).
Then, for each $T>0$\begin{align*}
\log W_{\beta_n,\lfloor T\beta_n^{-4}\rfloor }(\eta) \text{ is uniformly integrable.}
\end{align*}
\end{lem}

\begin{lem}\label{freeenecdp}
We have the limit \begin{align*}
F_\Z(\sqrt{2})=\lim_{T\to \infty}\frac{1}{T}P_\Z\left[\log \Z_{\sqrt{2}}(T)\right]=\sup_{T>0}\frac{1}{t}P_\Z\left[\log \Z_{\sqrt{2}}(T)\right]=-\frac{1}{6}.
\end{align*}
\end{lem}

We should take $n=\lfloor \beta_n^{-4}\rfloor$ in general. However, we may consider the case \begin{align*} \beta_n=n^{-\frac{1}{4}}\end{align*}
without loss of generality.

\subsection{Proof of limit superior}
We use the coarse graining argument to prove the limit superior.

We divide $\mathbb{Z}$ into the blocks with size of order $n^{1/2}$: For $y\in \mathbb{Z}$, we set \begin{align*}
B_y^n=\left[(2y-1)\lfloor n^{1/2}\rfloor+y,(2y+1)\lfloor n^{1/2}\rfloor+y\right].
\end{align*}

For each $\ell\in \mathbb{N}$, we denote  by $B_y^{n}(\ell)$ the set of lattice $z\in \mathbb{Z}$ such that \begin{align*}
z-\ell\in 2\mathbb{Z},
\end{align*}
that is the set of lattices in $B_y^n$ which can be reached by random walk $(S,P_S)$ at time $\ell$.

We will give an idea of the proof.
It is clear by Jensen's inequality that for each $\theta \in (0,1)$, $T\in\mathbb{N}$, and $N\in\mathbb{N}$,
\begin{align}
\begin{array}{ll}&\dis \frac{1}{NTn}Q\left[\log W_{\beta_n,NTn}(\eta)\right]\\
&\dis =\frac{1}{\theta NTn}Q\left[\log W^\theta_{\beta_n,NTn}(\eta)\right]\\
&\dis \leq \frac{1}{\theta NTn}\log Q\left[W^\theta_{\beta_n,NTn}(\eta)\right].\end{array}\label{thetaineq}
\end{align}
We will take the limit superior of both sides in $N\to\infty$, $n\to \infty$, $T\to \infty$, and then $\theta\to 0$ in this order. Then, it is clear that  \begin{align*}
\varlimsup_{n\to \infty}\frac{1}{\beta_n^4}F(\beta_n)\leq \varlimsup_{\theta\to 0}\varlimsup_{T\to \infty}\varlimsup_{n\to \infty}\varlimsup_{N\to\infty} \frac{1}{\theta NTn}\log Q\left[W^\theta_{\beta_n,NTn}(\eta)\right].
\end{align*}

We would like to estimate the right hand side.

\vspace{2em}

For $\theta \in (0,1)$, we have that \begin{align*}
Q\left[W^\theta_{\beta_n,NTn}(\eta)\right]\leq \sum_{Z}Q\left[\hat{W}_{\beta_n,NTn}^\theta(\eta,Z)\right],
\end{align*}
where for $Z=(z_1,\cdots,z_N)\in\mathbb{Z}^N$\begin{align*}
&\hat{W}_{\beta_n,NTn}(\eta,Z)\\
&=P_S\left[\prod_{i=1}^{NTn}\zeta_{i,S_i}(\beta_n,\eta):S_{\ell Tn}\in B_{z_i}^{n}(\ell Tn)\right],
\end{align*}
and we have used the fact  $(a+b)^\theta\leq a^\theta+b^\theta$ for $a,b\geq 0$ and $\theta\in (0,1)$. Then, we have from the Markov property that \begin{align}
Q\left[W^\theta_{\beta_n,NTn}(\eta)\right]&\leq \left(\sum_{z\in \mathbb{Z}}Q\left[\max_{x\in B_0^n(0)}P_S^x\left[\prod_{i=1}^{Tn}\zeta_{i,S_i}(\beta_n,\eta):S_{Tn}\in B_z^n\right]^\theta\right]\right)^N.\label{maxtheta}
\end{align}
Combining (\ref{thetaineq}) and (\ref{maxtheta}), we have that \begin{align*}
\frac{1}{\beta_n^4}F(\beta_n)\leq \frac{1}{\theta T}\log \sum_{z\in \mathbb{Z}}Q\left[\max_{x\in B_0^n(0)}P_S^x\left[\prod_{i=1}^{Tn}\zeta_{i,S_i}(\beta_n,\eta):S_{Tn}\in B_z^n\right]^\theta\right].
\end{align*}

Here, we have the following lemmas:

\begin{lem}\label{maxsup}
We have that \begin{align*}
\lim_{n\to\infty}Q\left[\max_{x\in B_0^n(0)}P_S^x\left[\prod_{i=1}^{Tn}\zeta_{i,S_i}(\beta_n,\eta)\right]^\theta\right]=P_\Z\left[\sup_{x\in [-1,1]}\left(\Z_{\sqrt{2}}^x(T)\right)^\theta\right].
\end{align*}
\end{lem}

\begin{lem}\label{partition1}
There exists a set $I^{(\theta)}(T)\subset \mathbb{Z}$ such that $\sharp I^{(\theta)}(T)\asymp T^2$ and for some constant $C_1>0$ and $C_2>0$, \begin{align*}
\sum_{z\in I^{(\theta)}(T)^c}Q\left[\max_{x\in B_0^n(0)}P_S^x\left[\prod_{i=1}^{Tn}\zeta_{i,S_i}(\beta_n,\eta):S_{Tn}\in B_z^n\right]^\theta\right]\leq C_1\exp(-C_2T^2)
\end{align*} 
for any $N\geq 1$.
\end{lem}

Then, we have that \begin{align*}
\varlimsup_{n\to \infty}\frac{1}{\beta_n}F(\beta_n)\leq \frac{1}{\theta T}\log  \left(C_1\exp(-C_2T^2)+C_3T^2 P_\Z\left[\sup_{x\in [-1,1]}\left(\Z_{\sqrt{2}}^x(T)\right)^\theta\right]\right).
\end{align*}

The following result gives us an upper bound of the limit superior:
\begin{lem}\label{ttheta}
We have that 
\begin{align*}
\varlimsup_{\theta\to 0}\varlimsup_{T\to \infty}\frac{1}{T\theta}\log P_\Z\left[\sup_{x\in [-1,1]}\left(\Z_{\sqrt{2}}^x(T)\right)^\theta\right]\leq F_\Z(\sqrt{2}).
\end{align*}
\end{lem}

\begin{rem}
(\ref{A}) is not assumed in lemmas in this subsection.  Thus, we find that the limit superior in Theorem \ref{main} is true for general environment.
\end{rem}

In the rest of the paper, we will prove the above lemmas.

\section{Proof of Lemma \ref{lpbdd}, Lemma \ref{maxsup}, and Lemma \ref{partition1}}\label{3}
\subsection{Proof of Lemma \ref{lpbdd}}

To prove  Lemma \ref{lpbdd}, we use the following concentration inequality.  

\begin{lem}\label{tal}(\cite[Lemma 3.3, Proposition 3.4]{CarTonTor})
Assume (\ref{A}). Then, for any $m\in\mathbb{N}$ and for any differentiable convex function $f:\mathbb{R}^m\to \mathbb{R}$, we have that \begin{align*}
Q\left(f(\eta)<a-t\right)Q\left(	f(\eta)>a, \left|\nabla f(\eta)\right|\leq c		\right)\leq C_1'\exp\left(-\frac{\left(\frac{t}{c}\right)^\gamma}{C_2'}\right),\ \ a\in\mathbb{R}, t,c\in(0,\infty),
\end{align*}
 where $\eta=\{\eta_1,\cdots,\eta_n\}$ are i.i.d.~random variables with the marginal law $Q(\eta_e\in dx)$ and $|\nabla f(\eta)|=\dis \sqrt{\sum_{i=1}^m \left|\frac{\partial }{\partial \eta_i}f(\eta)\right|^2}$.
\end{lem}

We take $\mathbb{R}^m$  as $\mathbb{R}^{E_n}$ in the proof of Lemma \ref{lpbdd} with $E_n=\{1,\cdots,Tn\}\times \{-Tn,\cdots,Tn\}$ which contains all lattices simple random walk can each up to time $Tn$.

\begin{proof}[Proof of Lemma \ref{lpbdd}] When we look at $W_{\beta_n,Tn}(\eta)$ as the function of $\{\eta(i,x):(i,x)\in E_n\}$, $\log W_{\beta_n,Tn}(\eta)$ is differentiable and convex. Indeed, we have that \begin{align*}
\frac{\partial }{\partial \eta(i,x)}\log W_{\beta_n,Tn}(\eta)=\frac{1}{W_{\beta_n,Tn}(\eta)}P_S\left[\beta_n \prod_{k=1}^{Tn}\zeta_{k,S_k}(\beta_n,\eta): S_i=x \right]
\end{align*}
and for $s\in [0,1]$, for $\eta=\{\eta(i,x):(i,x)\in E_n\}$ and $\eta'=\{\eta'(i,x):(i,x)\in E_n\}$\begin{align*}
W_{\beta_n,Tn}(s\eta+(1-s)\eta')&=P_S\left[\prod_{k=1}^{Tn}\zeta_{k,S_k}(\beta_n,s\eta+(1-s)\eta')\right]\\
&\leq W_{\beta_n,Tn}(\eta)^sW_{\beta_n,Tn}(\eta')^{1-s}.
\end{align*}

Thus, we can apply Lemma \ref{tal} to $\log W_{\beta_n,Tn}(\eta)$. Since \begin{align*}
\left|\nabla \log W_{\beta_n,Tn}(\eta)\right|^2&=\beta_n^2\sum_{(i,x)\in E_n}\left(\frac{1}{W_{\beta_n,Tn}}P_S\left[\prod_{k=1}^{Tn}\zeta_{k,S_k}(\beta_n,\eta):S_i=x\right]\right)^2\\
&=\beta_n^2\sum_{(i,x)\in E_n}\mu_{\beta_n,Tn}^{\eta}(S_i=x)^2\\
&=\beta_n^2\left(\mu_{\beta_n,Tn}^{\eta}\right)^{\otimes 2}\left[\sharp\{1\leq i\leq Tn:S_i=S_i'\}\right],
\end{align*}
where $\mu_{\beta,n}^\eta$ is the probability measure on the simple random walk paths defined by \begin{align*}
\mu_{\beta,n}^\eta(s)=\frac{1}{W_{\beta,n}(\eta)}\exp\left(\beta \sum_{i=1}^N\eta(i,s_i)-n\lambda(\beta)\right)P_S(S_{[0,n]}=s),\ \ \ s=(s_0,s_1,\cdots,s_n)\in\mathbb{Z}^{n+1},
\end{align*}
with $\left(\mu_{\beta_n,Tn}^{\eta}\right)^{\otimes 2}$ is the product probability measure of $\mu_{\beta_n,Tn}^{\eta}$, and $S$ and $S'$ are paths of independent directed path with the law $\mu_{\beta_n,Tn}^{\eta}$.

We write \begin{align*}
L_n(s,s')=\sharp\{1\leq i\leq n:s_i=s_i'\}
\end{align*}
for $s=(s_1,\cdots,s_n)$ and $s'=(s_1',\cdots,s_n')\in \mathbb{Z}^n$.


We define the event $A_n$ on the environment by \begin{align*}
&A_n\\
&=\left\{\eta: W_{\beta_n,Tn}(\eta)\geq \frac{1}{2}Q\left[W_{\beta_n,Tn}(\eta)\right],\  \beta_n^2\left(\mu_{\beta_n,Tn}^{\eta}\right)^{\otimes 2}\left[L_{Tn}(S,S')\right]\leq C_4\right\}
\end{align*}
for some $C_4>0$ which we will take large enough. We claim that for  $C_4>0$ large enough, there exists a constant $\delta>0$ such that \begin{align}
Q(A_n)>\delta\label{lbdd}
\end{align}
for all $n\geq 1$. If (\ref{lbdd}) holds, then Lemma \ref{lpbdd} follows.

Indeed, applying Lemma \ref{tal}, we have that  \begin{align*}
Q(\log W_{\beta_n,Tn}(\eta)\leq -\log 2 -u)\leq Q(A_n)^{-1}C_1'\exp\left(-\frac{\left(\frac{u}{C_1'}\right)^\gamma}{C_2'}\right).
\end{align*}
 Thus, we find the $L^2$-boundedness of $\log W_{\beta_n,Tn}(\eta)$ and hence uniform integrability.

We will complete the proof of Lemma \ref{lpbdd} by showing that (\ref{lbdd}). We observe that \begin{align*}
&Q(\eta\in A_n)\\
&=Q\left(\left\{\eta: W_{\beta_n,Tn}(\eta)\geq \frac{1}{2}Q\left[W_{\beta_n,Tn}(\eta)\right]\right\}\right)\\
&-Q\left(\left\{\eta: W_{\beta_n,Tn}(\eta)\geq \frac{1}{2}Q\left[W_{\beta_n,Tn}(\eta)\right],\ \ \left(\mu_{\beta_n,Tn}^{\eta}\right)^{\otimes 2}\left[L_{\lfloor Tn}(S,S')\right]> C_4\right\} \right)\\
&\geq Q\left(\left\{\eta: W_{\beta_n,Tn}(\eta)\geq \frac{1}{2}\right\}\right)\\
 &-Q\left(\left\{\eta: P_{S,S'}\left[\beta_n^2L_{Tn}(S,S')\prod_{i=1}^{Tn}\zeta_{i,S_i}(\beta_n,\eta)\zeta_{i,S_i}(\beta_n,\eta)\right]> \frac{C_4}{4}\right\} \right)\\
 &\geq Q\left(\left\{\eta: W_{\beta_n,Tn}(\eta)\geq \frac{1}{2}Q\left[W_{\beta_n,Tn}(\eta)\right]\right\}\right)\\
 &-\frac{4}{C_4}Q\left[P_{S,S'}\left[\beta_{n}^2L_{Tn}(S,S')\prod_{i=1}^{Tn}\zeta_{i,S_i}(\beta_n,\eta)\zeta_{i,S_i'}(\beta_n,\eta)\right]\right],
\end{align*}
where $(S',P_{S'})$ is the simple random walk on $\mathbb{Z}$ starting from the origin and $P_{S,S'}$ is the product measure of $P_S$ and $P_{S'}$.

Paley-Zygmund's inequality yields that \begin{align*}
Q\left(\left\{\eta: W_{\beta_n,Tn}(\eta)\geq \frac{1}{2}Q\left[W_{\beta_n,Tn}(\eta)\right]\right\}\right)&\geq \frac{1}{4}\frac{\left(Q\left[W_{\beta_n,Tn}(\eta)\right]\right)^2}{Q\left[W_{\beta_n,Tn}(\eta)^2\right]}\\
&=\frac{1}{4}\frac{1}{Q\left[W_{\beta_n,Tn}(\eta)^2\right]}.
\end{align*}
Also, we have that \begin{align*}
&Q\left[P_{S,S'}\left[\beta_{n}^2L_{Tn}(S,S')\prod_{i=1}^{Tn}\zeta_{i,S_i}(\beta_n,\eta)\zeta_{i,S_i'}(\beta_n,\eta)\right]\right]\\
&=P_{S,S'}\left[\beta_n^2L_{Tn}(S,S')\exp\left(\left(\lambda(2\beta_n)-2\lambda(\beta_n)\right)L_{Tn}(S,S')\right)\right]\\
&\leq P_{S,S'}\left[\exp\left(2\left(\lambda(2\beta_n)-2\lambda(\beta_n)\right)L_{Tn}(S,S')\right)\right].
\intertext{and}
&Q\left[W_{\beta_n,Tn}(\eta)^2\right]\\
&=Q\left[P_{S,S'}\left[\prod_{i=1}^{Tn}\zeta_{i,S_i}(\beta_n,\eta)\zeta_{i,S_i'}(\beta_n,\eta)\right]\right]\\
&=P_{S,S'}\left[\exp\left(2\left(\lambda(2\beta_n)-2\lambda(\beta_n)\right)L_{Tn}(S,S')\right)\right].
\end{align*}

Since \begin{align*}
\frac{\lambda(2\beta_n)-2\lambda(\beta_n)}{\beta_n^2}\to \lambda''(0)=1
\intertext{and}
\frac{\lambda(2r\beta_n)-2\lambda(r\beta_n)}{\lambda(2\beta_n)-2\lambda(\beta_n)}\to r^2
\end{align*}
as $n\to \infty$ for $r>0$, there exists a constant $r>0$ such that \begin{align*}
&P_{S,S'}\left[\exp\left(2\left(\lambda(2\beta_n)-2\lambda(\beta_n)\right)L_{Tn}(S,S')\right)\right]\\
&\leq P_{S,S'}\left[\exp\left(\left(\lambda(2r\beta_n)-2\lambda(r\beta_n)\right)L_{Tn}(S,S')\right)\right]\\
&\leq Q\left[W_{r\beta_n,Tn}(\eta)^2\right]
\end{align*}
for any $n$ large enough. The $L^2$-boundedness of $W_{\beta_n,Tn }(\eta)$ (see Theorem \ref{thmref}) implies that there exist $C_5>0$ and $C_6>0$ such that \begin{align*}
&Q\left[W_{\beta_n,Tn}(\eta)^2\right]\leq C_5\\
\intertext{and}
&Q\left[P_{S,S'}\left[\beta_{n}^2L_{Tn}(S,S')\prod_{i=1}^{Tn}\zeta_{i,S_i}(\beta_n,\eta)\zeta_{i,S_i'}(\beta_n,\eta)\right]\right]\leq C_6.
\end{align*}

We conclude that \begin{align*}
Q(\eta\in A_n)\geq \frac{1}{4C_{5}}-\frac{4C_{6}}{C_4}
\end{align*}
and we obtain (\ref{lbdd}) by taking $C_4>0$ large enough.
\end{proof}

\subsection{Proof of Lemma \ref{maxsup}}

Since the finite dimensional distributions $\dis \left\{W_{\beta_n,Tn}^{x_in^{1/2}}(\eta):1\leq i\leq m\right\}$ for $x_1,\cdots,x_m\in B_0^n(0)$ converge to 
$\dis \left\{\Z_{\sqrt{2}}^{x_i}(T):1\leq i\leq m\right\}$ (see \cite[Section 6.2]{AlbKhaQua}), the  tightness of $\{W_{\beta_n,Tn}^{xn^{1/2}}(\eta):x\in [-1,1]\}$ in $C[-1,1]$ and $L^p$-boundedness of $ \max_{x\in[-1,1]}W_{\beta_n,Tn}^{xn^{1/2}}(\eta)^\theta$ for $p>1$ imply Lemma \ref{maxsup}.


We will use Garsia-Rodemich-Rumsey's lemma \cite[Lemma A.3.1]{Nua} a lot of times in the proof for limit superior.
\begin{lem}\label{Gar}
Let $\phi:[0,\infty)\to [0,\infty)$ and $\Psi:[0,\infty)\to [0,\infty)$ be continuous and stricctly increasing functions satisfying\begin{align*}
\phi(0)=\Psi(0)=0,\ \lim_{t\to \infty}\Psi(t)=\infty.
\end{align*}
Let $f:\mathbb{R}^d\to \mathbb{R}$ be a continuous function. Provided \begin{align*}
\Gamma=\int_{B_r(x)}\int_{B_r(x)}\Psi\left(\frac{|f(t)-f(s)|}{\phi(|t-s|)}\right)dsdt<\infty,
\end{align*}
where $B_r(x)$ is an open ball in $\mathbb{R}^d$ centered at $x$ with radius $r$, then for all $s,t\in B_r(x)$,\begin{align*}
|f(t)-f(s)|\leq 8\int_0^{2|t-s|}\Psi^{-1}\left(\frac{4^{d+1}\Gamma}{\lambda_d u^{2d}}\right)\phi(du),
\end{align*}
where $\lambda_d$ is a universal constant depending only on $d$.
\end{lem}

Applying Lemma \ref{Gar} with $\Psi(x)=|x|^p$, $\phi(u)=u^q$ for $p\geq 1$, $q>0$ and $pq>2d$, we have that \begin{align}
|f(t)-f(s)|\leq \frac{2^{\frac{2}{p}+q+3}}{\lambda_d^{\frac{1}{p}}\left(q-\frac{2d}{p}\right)}|t-s|^{q-\frac{2d}{p}}\left(\int_{B_1(x)}\int_{B_1(x)}\left(\frac{|f(t)-f(s)|}{|t-s|^{q}}\right)^pdsdt\right)^{\frac{1}{p}}\label{Garpol}
\end{align}
for $t,s\in B_1(x)$.

We set \begin{align*}
&W_{\beta,n}^x(\eta)=P_S^x\left[\prod_{k=1}^{n}\zeta_{k,S_k}(\beta,\eta)\right]
\intertext{and}
&W_{\beta,n}^{x}(\eta,y)=P_S^x\left[\prod_{k=1}^n\zeta_{k,S_k}(\beta,\eta):S_n=y\right]
\end{align*}
for $x,y\in\mathbb{Z}$.

For $-1\leq u\leq 1$, we define \begin{align*}
f_{n,\theta}(u)=\begin{cases}
\left(W_{\beta_n,Tn}^{un^{1/2}}(\eta)\right)^\theta,\ \ \ \ &un^{1/2}\in B_0^{n}(0)\\
\text{linear  interpolation,}\  &\text{otherwise}.
\end{cases}
\end{align*}

Then, we have that \begin{align*}
\max_{x\in B_0^n(0)}\left(W_{\beta_n,Tn}^x(\eta)\right)^\theta\leq \left(W_{\beta_n,Tn}(\eta)\right)^\theta+C_{p,q}B_{p,q,n,\theta},
\end{align*}
where \begin{align*}
B_{p,q,n,\theta}=\left(\int_{-1}^1\int_{-1}^1\left(\frac{|f_{n,\theta}(t)-f_{n,\theta}(s)|}{|t-s|^{q}}\right)^pdsdt\right)^{\frac{1}{p}}.
\end{align*}

We will show that for some $p\geq 1$, $q>0$ with $pq>2$, there exist $C_{p,T,\theta}>0$ and $\eta_{p,\theta}-pq>-1$ such that \begin{align}
&Q\left[|f_{n,\theta}(t)-f_{n,\theta}(s)|^p\right]\leq C_{p,T,\theta}|t-s|^{\eta_{p,\theta}},\ \ -1\leq s,t\leq 1.\label{contcri}
\end{align}

(\ref{contcri}) tells us the tightness of $\{W_{\beta_n,Tn}^{xn^{1/2}}(\eta):x\in [-1,1]\}$ in $C[-1,1]$  and $L^p$-boundedness of $ \max_{x\in[-1,1]}W_{\beta_n,Tn}^{xn^{1/2}}(\eta)^\theta$ for $p>1$ and therefore Lemma \ref{maxsup} follows.

\begin{proof}[Proof of (\ref{contcri})]
We remark that \begin{align*}
|f_{n,\theta}(t)-f_{n,\theta}(s)|\leq  \left|W_{\beta_n,Tn}^{tn^{1/2}}(\eta)-W_{\beta_n,Tn}^{sn^{1/2}}(\eta)\right|^{\theta},
\end{align*}
where we have used that $(x+y)^\theta \leq x^\theta+y^\theta$ for $x\geq 0,$ $y\geq 0$. First, we will estimate \begin{align*}
Q\left[\left|W_{\beta_n,Tn}^{x}(\eta)-W_{\beta_n,Tn}^{y}(\eta)\right|^2\right],
\end{align*}
for $x,y\in B_0^n(0)$.
When we define i.i.d.\,random variables by\begin{align*}
e_n(k,x)=\exp\left(\beta_n\eta(k,x)-\lambda(\beta_n)\right)-1,\ \ \ (k,x)\in\mathbb{N}\times \mathbb{Z},
\end{align*}
we find that \begin{align*}
Q[e_n(k,x)]=0,\ \ \text{and }\frac{Q\left[e_n(k,x)^2\right]}{\beta_n^2}=\frac{e(\lambda(2\beta_n)-2\lambda(\beta_n))-1}{\beta_n^2}\to 1.
\end{align*}

Then, we can write \begin{align*}
W_{\beta_n,Tn}^{x}(\eta)&=P_S^x\left[\prod_{i=1}^{Tn}\left(1+e_n(i,S_i)\right)\right]\\
&=1+\sum_{k=1}^{Tn} \sum_{1\leq i_1<\cdots <i_k\leq Tn}\sum_{\mathbf{x}\in\mathbb{Z}^k}\prod_{j=1}^kp_{i_j-i_{j-1}}(x_j-x_{j-1})e_n(i_j,x_j)\\
&=\sum_{k=0}^{Tn}\Theta^{(k)}(x),
\end{align*}
where $p_n(y)=P_S(S_n=y)$ for $(n,y)\in\mathbb{N}\times \mathbb{Z}$, $x_0=x$, $\mathbf{x}=(x_1,\cdots,x_k)$, and 
\begin{align*}
\Theta^{(k)}(x)=\begin{cases}
\dis 1,\ \ \ &k=0\\
\dis \sum_{1\leq i_1<\cdots <i_k\leq Tn}\sum_{\mathbf{x}\in\mathbb{Z}^k}\prod_{j=1}^kp_{i_j-i_{j-1}}(x_j-x_{j-1})e_n(i_j,x_j), \ \ &k\geq 1.
\end{cases}
\end{align*}

Then, it is easy to see that \begin{align*}
&Q\left[\Theta^{(k)}(x)\right]=0,\ \ k\geq 1
\intertext{and}
&Q\left[\Theta^{(k)}(x)\Theta^{(\ell)}(y)\right]=0,\ \ k\not=\ell,\ \ x,y\in\mathbb{Z}.
\end{align*}
 Thus, we have that \begin{align*}
&Q\left[\left|W_{\beta_n,Tn}^{x}(\eta)-W_{\beta_n,Tn}^{y}(\eta)\right|^2\right]\\
&=\sum_{k=1}^{Tn}Q\left[\left(\Theta^{(k)}(x)-\Theta^{(k)}(y)\right)^2\right]\\
&=\sum_{k=1}^{Tn}\left(Q[e_{n}(0,0)^2]\right)^k\sum_{1\leq i_1<\cdots <i_k\leq Tn} \sum_{\mathbf{x}\in\mathbb{Z}^k}(p_{j_1}(x_1-x)-p_{j_1}(x_1-y))^2\prod_{j=2}^kp_{i_j-i_{j-1}}(x_j-x_{j-1})^2.
\end{align*}

Since we know that for $k\geq 1$\begin{align*}
\frac{1}{n^{k/2}}\sum_{1\leq i_1<\cdots <i_k\leq Tn} \sum_{\mathbf{x}\in\mathbb{Z}^k}\prod_{j=1}^kp_{i_j-i_{j-1}}(x_j-x_{j-1})^2\leq \frac{C_7^kT^{k/2}}{\Gamma\left(\frac{k}{2}+1\right)},
\end{align*}
where $\Gamma(s)$ is a Gamma function at $s>0$ \cite[Section 3.4 and Lemma A.1]{AlbKhaQua},
\begin{align*}
Q\left[\left(\Theta^{(k)}(x)-\Theta^{(k)}(y)\right)^2\right]&\leq \sum_{z\in\mathbb{Z}}(p_{i}(z-x)-p_i(z-y))^2 Q[e_n(0,0)^2]\frac{C_7^{k-1}T^{\frac{k-1}{2}}}{\Gamma\left(\frac{k-1}{2}+1\right)}\\
&=Q[e_n(0,0)^2]\frac{C_7^{k-1}T^{\frac{k-1}{2}}}{\Gamma\left(\frac{k-1}{2}+1\right)}\sum_{1\leq i\leq Tn}(p_{2i}(0)-p_{2i}(x-y)).
\end{align*}

Since we know that there exists a constant $c>0$ such that for all $n\geq 1$ and $x,y\in\mathbb{Z}$ with $x-y\in 2\mathbb{Z}$,\begin{align}
|p_{n}(x+y)-p_{n}(x)|&\leq \frac{c|x|}{{n^{2}}}+2\left(\frac{1}{\sqrt{4\pi n}}\right)\left|\exp\left(-\frac{(x+y)^2}{4n}\right)-\exp\left(-\frac{x^2}{4n}\right)\right|\label{LLT}
\end{align}
(see \cite[Theorem 2.3.6]{LawLim}), we have that 
\begin{align}
Q\left[\left(\Theta^{(k)}(x)-\Theta^{(k)}(y)\right)^2\right]&\leq C\left|\frac{x-y}{\sqrt{n}}\right|\frac{C_7^{k-1}T^{\frac{k-1}{2}}}{\Gamma\left(\frac{k-1}{2}+1\right)}\label{theta2}
\end{align}
and
\begin{align*}
Q\left[\left|W_{\beta_n,Tn}^{x}(\eta)-W_{\beta_n,Tn}^{y}(\eta)\right|^2\right]\leq C_{1,T}\left|\frac{x-y}{\sqrt{n}}\right|,
\end{align*}
{for }$x,y\in B_0^n(0)$, where we remark that \begin{align*}
C_{1,T}=C\sum_{k\geq 1}\frac{C_7^{k-1}T^{\frac{k-1}{2}}}{\Gamma(\frac{k-1}{2}+1)}
\end{align*}
and \begin{align}
\varlimsup_{T\to\infty} \frac{1}{T}\log C_{1,T}\leq C_{8}.\label{C3}
\end{align}

Now, we would like to estimate  
\begin{align*}
Q\left[\left|W_{\beta_n,Tn}^{tn^{1/2}}(\eta)-W_{\beta_n,Tn}^{sn^{1/2}}(\eta)\right|^p\right]
\end{align*}
for $p\geq 2$, $s,t\in [-1,1]$ with $sn^{1/2}, tn^{1/2}\in B_0^n(0)$. 

Then, the hypercontractivity established in \cite[Proposition 3.11, Proposition 3.12, and Proposition 3.17]{MosOdoOle} allows us to estimate 
\begin{align*}
Q\left[\left|W_{\beta_n,Tn}^{tn^{1/2}}(\eta)-W_{\beta_n,Tn}^{sn^{1/2}}(\eta)\right|^p\right].
\end{align*}
Indeed, \begin{align*}
&Q\left[\left|W_{\beta_n,Tn}^{tn^{1/2}}(\eta)-W_{\beta_n,Tn}^{sn^{1/2}}(\eta)\right|^p\right]^{1/p}\\
&=Q\left[\left|\sum_{k=1}^{Tn}\left(\Theta^{(k)}(tn^{1/2})-\Theta^{(k)}(sn^{1/2})\right)\right|^p\right]^{1/p}\\
&\leq \left(\sum_{k=1}^{Tn}Q\left[\left|\Theta^{(k)}(tn^{1/2})-\Theta^{(k)}(sn^{1/2})\right|^p\right]\right)^{1/p}\\
&\leq \left(\sum_{k=1}^{Tn}\kappa_p^k\left(Q\left[\left(\Theta^{(k)}(tn^{1/2})-\Theta^{(k)}(sn^{1/2})\right)^2\right]^{1/2}\right)\right)^{1/p},
\end{align*}
where $\dis \kappa_p=2\sqrt{p-1}\sup_{n\geq 1}\frac{Q[e_n(0,0)^p]^{1/p}}{Q[e_n(0,0)^2]^{1/2}}<\infty$. $\kappa_p$ is finite since \begin{align*}
\lim_{n\to \infty}\frac{1}{\beta_n}Q\left[|e_n(0,0)|^p\right]^{1/p}=Q\left[|\eta(0,0)|^p\right]^{1/p}.
\end{align*}
We obtain from (\ref{theta2}) \begin{align*}
Q\left[\left|W_{\beta_n,Tn}^{tn^{1/2}}(\eta)-W_{\beta_n,Tn}^{sn^{1/2}}(\eta)\right|^p\right]&\leq C|t-s|^{\frac{p}{2}}\sum_{k\geq 1}\kappa_p^k\left(\frac{C_7^{k-1}T^{\frac{k-1}{2}}}{\Gamma(\frac{k-1}{2}+1)}\right)^{1/2}\\
&\leq C_{p,T}|t-s|^{\frac{p}{2}}.
\end{align*}

Thus, we find that for  $p\geq \frac{2}{\theta}$ \begin{align*}
\eta_{p,\theta}=\frac{p\theta}{2}
\end{align*}in (\ref{contcri}). Therefore, the proof completed when we take $p=\dis \frac{5}{\theta}$ and $q=\dis \frac{2\theta}{3}$.
\end{proof}

\subsection{Proof of Lemma \ref{partition1}}

The idea is the same as the proof of Lemma \ref{maxsup}.

\begin{proof}[Proof of Lemma \ref{partition1}] We set \begin{align*}
&W_{\beta,n}^{x}(\eta,A)=\sum_{y\in A}P_S^x\left[\prod_{k=1}^n\zeta_{k,S_k}(\beta,\eta):S_n=y\right]
\end{align*}
for $A\subset \mathbb{Z}$.

Then, we know that \begin{align*}
&\left|\left(W_{\beta_n,Tn}^x(\eta,B_z^n)\right)^\theta-\left(W_{\beta_n,Tn}^y(\eta,B_z^n)\right)^\theta\right|\\
&\leq \left|W_{\beta_n,Tn}^x(\eta,B_z^n)-W_{\beta_n,Tn}^y(\eta,B_z^n)\right|^\theta.
\end{align*}
By the same argument as the proof of Lemma \ref{maxsup} that \begin{align*}
&Q\left[\max_{x\in B_0^n}\left(\sum_{w\in B_z^n}W_{\beta_n,Tn}^x(\eta,w)\right)^\theta\right]\\
&\leq Q\left[\left(W_{\beta_n,Tn}(\eta,B_z^n)\right)^\theta\right]+C_{p,q}B_{p,q,n,\theta,z,T}\\
&\leq \left(\sum_{y\in B_z^n}p_{Tn}(y)\right)^{\theta}+C_{p,q}B_{p,q,n,\theta,z,T},
\end{align*}
where \begin{align*}
B^p_{p,q,n,\theta,z,T}=\int_{-1}^1\int_{-1}^1 \frac{Q\left[\left|W_{\beta_n,Tn}^{tn^{1/2}}(\eta,B_z^n)-W_{\beta_n,Tn}^{tn^{1/2}}(\eta,B_z^n)\right|^{p\theta}\right]}{|t-s|^{pq}}dsdt.
\end{align*}
We write \begin{align*}
&W_{\beta_n,Tn}^{x}(\eta,w)\\
&=P_{S}^x\left[\prod_{i=1}^{Tn}(1+e_n({i,S_i})):S_{Tn}=w\right]\\
&=p_{Tn}(w-x)\\
&+p_{Tn -i_{k}}(w-x_k)\sum_{k=1}^{Tn}\sum_{1\leq  i_1<\cdots<i_k\leq Tn}\sum_{{\bf x}\in \mathbb{Z}^k}\left(\prod_{i=1}^k p_{i_j-i_{j-1}}(x_i-x_{j-1})e_n(i_{j},x_j)\right)\\
&=\sum_{k=0}^{Tn}\Theta^{(k)}(x,w),
\end{align*}
where \begin{align*}
&\Theta^{(k)}(x,w)\\
&=\begin{cases}
\dis p_{Tn}(x,w),\ \ &k=0\\
\dis p_{Tn -i_{k}}(w-x_k)\sum_{k=1}^{Tn}\sum_{1\leq  i_1<\cdots<i_k\leq Tn}\sum_{{\bf x}\in \mathbb{Z}^k}\left(\prod_{i=1}^k p_{i_j-i_{j-1}}(x_i-x_{j-1})e_n(i_{j},x_j)\right),\ \ &k\geq 1.
\end{cases}
\end{align*}

Then, we have that \begin{align*}
&Q\left[\Theta^{(k)}(x,w)\right]=0,\ \ \ k\geq 1\\
&Q\left[\Theta^{(k)}(x,y)\Theta^{(\ell)}(z,w)\right]=0,\ \ k\not=\ell.
\end{align*}
Hence, 
\begin{align*}
Q\left[\left(\sum_{w\in B_z^n}(\Theta^{(0)}(x,w)-\Theta^{(0)}(y,w))\right)^2\right]& =\left(\sum_{w\in B_z^n}\left(p_{Tn}(w-x)-p_{Tn}(w-y)\right)\right)^2.
\end{align*}
(\ref{LLT}) implies that \begin{align*}
Q\left[\left(\sum_{w\in B_z^n}(\Theta^{(0)}(tn^{1/2},w)-\Theta^{(0)}(sn^{1/2},w))\right)^2\right]& =C_{2,T}|t-s|,\ \ \ t,s\in[0,1],
\end{align*}
where $C_{2,T}\to 0 $ as $T\to \infty$.
Also, we have that for $k\geq 1$\begin{align*}
&Q\left[\left(\sum_{w\in B_z^n}(\Theta^{(k)}(x,w)-\Theta^{(k)}(y,w))\right)^2\right]\\
&=Q[e_{n}(0,0)^2]^k\sum_{1\leq i_1<\cdots<i_k\leq Tn}\sum_{\x\in\mathbb{Z}^k}\\
&\ \ \ (p_{i_1}(x_1-x)-p_{i_1}(x_1-y))^2\left(\prod_{i=2}^kp_{i_{j}-i_{j-1}}(x_j-x_{j-1})^2\right)^2\left(\sum_{w\in B_{z}^n}p_{\lfloor T_n\rfloor-i_k}(w-x_{k})\right)^2\\
&\leq Q[e_{n}(0,0)^2]^k\sum_{1\leq i_1<\cdots<i_k\leq Tn}(p_{2i_1}(0)-p_{2i_1}(x-y))^2\prod_{i=2}^kp_{2(i_{j}-i_{j-1})}(0)\\
&\leq C\frac{|x-y|}{n^{1/2}}\frac{C_7^{k-1}T^{\frac{k-1}{2}}}{\Gamma\left(\frac{k-1}{2}+1\right)}
\end{align*}
as the proof of Lemma \ref{maxsup}.

We obtain by H\"older's inequality  that for $p=\frac{5}{\theta}$\begin{align*}
&Q\left[\left|W_{\beta_n,Tn}^{tn^{1/2}}(\eta,B_z^n)-W_{\beta_n,Tn}^{tn^{1/2}}(\eta,B_z^n)\right|^{p\theta}\right]\\
&\leq Q\left[\left|W_{\beta_n,Tn}^{tn^{1/2}}(\eta,B_z^n)-W_{\beta_n,Tn}^{tn^{1/2}}(\eta,B_z^n)\right|^{\frac{9}{2}}\left|W_{\beta_n,Tn}^{tn^{1/2}}(\eta,B_z^n)+W_{\beta_n,Tn}^{tn^{1/2}}(\eta,B_z^n)\right|^{\frac{1}{2}}\right]\\
&\leq Q\left[\left|W_{\beta_n,Tn}^{tn^{1/2}}(\eta,B_z^n)-W_{\beta_n,Tn}^{tn^{1/2}}(\eta,B_z^n)\right|^{9}\right]^{\frac{1}{2}}Q\left[W_{\beta_n,Tn}^{tn^{1/2}}(\eta,B_z^n)+W_{\beta_n,Tn}^{tn^{1/2}}(\eta,B_z^n)\right]^{\frac{1}{2}}\\
&\leq C_{3,T}|t-s|^{9/2}\left(2\sum_{w\in B_z^n}(p_{Tn}(tn^{1/2},w)+p_{Tn}(sn^{1/2},w))\right)^\frac{1}{2},
\end{align*}
where we have used the hypercontractivity as the proof of Lemma \ref{maxsup}, $C_{3,T}$ is independent of the choice of $z$ and \begin{align*}
\varlimsup_{T\to\infty}\frac{1}{T}\log C_{3,T}\leq C<\infty.
\end{align*}
Also, we know that 
\begin{align*}
\sum_{w\in B_z^n}p_{Tn}(x,w)\leq \exp\left(-\frac{z^2n}{Tn}\right)
\end{align*}
for $x\in B_0^n$. Thus, we obtain that if $I^{(\theta)}(T)\asymp T^2$ with $p=\frac{5}{\theta}$, $q=\frac{\theta}{2}$, there exist $C_1>0$ and $C_2>0$ such that \begin{align*}
\sum_{z\in I^{(\theta)}(T)^c} \left(\left(\sum_{y\in B_z^n}p_{Tn}(y)\right)^{\theta}+C_{p,q}B_{p,q,n,\theta,z,T}\right)\leq C_1\exp\left(-C_2T^2\right).
\end{align*}

\end{proof}


\section{Continuum directed polymers}\label{4}
To prove Lemma \ref{freeenecdp} and Lemma \ref{ttheta}, we recall the property of continuum directed polymers.

\subsection{Continuum directed polymers}
The mild solution to stochastic heat equation \begin{align*}
\partial_t \Z=\frac{1}{2}\Delta \Z+\beta \Z \dot{\W},\ \ \lim_{t\searrow 0}\Z(t,y)=\delta_x(y) 
\end{align*}
has the following representation using Wiener chaos expansion:\begin{align*}
\Z^x_\beta(T,w)&=\rho_T(x,w)\\
&+\sum_{n\geq 1}\beta^n\int_{\Delta_n(T)}\int_{\mathbb{R}^n}\left(\prod_{i=1}^n\rho_{t_i-t_{i-1}}(x_i-x_{i-1})\right)\rho_{T-t_n}(w-x_n)\W(dt_1dx_1)\cdots \W(dt_ndx_n),
\end{align*}
where we set $x_0=x$\begin{align*}
\rho_t(x,w)=\rho_t(x-w)=\frac{1}{\sqrt{2\pi t}}\exp\left(-\frac{(x-w)^2}{2t}\right),\ \  \ t>0,\ x,w\in\mathbb{R},
\end{align*}
and \begin{align*}
\Delta_n(T)=\{(t_1,\cdots,t_n):0< t_1< \cdots< t_n\leq T\}.
\end{align*}

Also, we define the four parameter field by \begin{align*}
\Z_\beta(s,x;t,y)&=\rho_{t-s}(x,y)\\
&\hspace{-3em}+\sum_{n\geq 1}\beta^n\int_{\Delta_n(s,t)}\int_{\mathbb{R}^n}\left(\prod_{i=1}^n\rho_{t_i-t_{i-1}}(x_i-x_{i-1})\right)\rho_{t-t_n}(y-x_n)\W(dt_1dx_1)\cdots \W(dt_ndx_n),
\end{align*}
for $ 0\leq s<t<\infty,\ \ x,y\in\mathbb{R}^2$, where we set $t_0=s$ and \begin{align*}
\Delta_n(s,t)=\{(t_1,\cdots,t_n):s< t_1< \cdots< t_n\leq t\}.
\end{align*}

Also, we define  \begin{align*}
\Z_{\beta}^{(s,x)}(t)=\int_{\mathbb{R}}\Z_{\beta}(s,x;t,y)dy, \ \ \text{for }0\leq s<t<\infty,\ \ x\in\mathbb{R}.
\end{align*}

Then, we have the following fact\cite[Theorem 3.1]{AlbKhaQua2}:
\begin{thm}\label{cdpprop}There exists a version of the field $\Z_\beta(s,x;t,y)$ which is jointly continuous in all four variables and have the following  properties:
\begin{enumerate}[(i)]
\item $P_\Z\left[\Z_\beta(s,x;t,y)\right]=\rho_{t-s}(y-x)$.
\item ({\bf Stationary}): $\dis \Z_\beta(s,x;t,y)\stackrel{d}{=}\Z_\beta(s+u_0,x+z_0;t+u_0,y+z_0)$.
\item  ({\bf Scaling}): $\dis \Z_\beta(r^2s,rx;r^2t,ry)\stackrel{d}=\frac{1}{r}\Z_{\beta\sqrt{r}}(s,x;t,y)$.
\item ({\bf Positivity}): With probability one, $\Z_\beta(s,x;t,y)$ is strictly positive for all tuples $(s,x;t,y)$ with $0\leq s<t$.
\item The law of $\dis \frac{\Z_\beta(s,x;t,y)}{\rho_{t-s}(y-x)}$ does not depend on $x$ or $y$.
\item It has an independent property among disjoint time intervals: for any finite $\{(s_1,t_i]\}_{i=1}^n$ and any $x_i,y_i\in\mathbb{R}$, the random variables $\{\Z_\beta(s_i,x_i;t_i,y_i)\}_{i=1}^n$ are mutually independent.
\item ({\bf Chapman-Kolmogorov equations}): With probability one, for all $0\leq s<r<t$ and $x,y\in\mathbb{R}$,\begin{align*}
\Z_\beta({s,x;t,y})=\int_{\mathbb{R}}\Z_\beta(s,x;r,z)\Z_\beta(r,z;t,y)dz.
\end{align*}
\end{enumerate}

\end{thm}

The following is the corollary of \cite[Theorem 1.1]{AmiCorQua}.
 \begin{thm}\label{limitcdp}
$\dis \frac{1}{T}\log \Z_1(T,0)$ converges to $\dis -\frac{1}{4!}$ in probability as $T\to \infty$.
\end{thm}

Also, the following is the result obtained by Moreno \cite{Mor}:

\begin{cor}\label{cor}
For any $\beta\geq 0$ and $p\geq 1$, $\left(\Z_{\sqrt{2}}(t)\right)^{-1}\in L^p$.
\end{cor}

\subsection{Proof of Lemma \ref{ttheta}}

We first show a weak statementt:
\begin{lem}\label{wttheta}
We have that \begin{align*}
\varlimsup_{\theta\to 0}\varlimsup_{T\to\infty}\frac{1}{T\theta}P_\Z\left[\left(\Z_{\sqrt{2}}(T)\right)^\theta\right]\leq F_\Z(\sqrt{2}).
\end{align*}
\end{lem}

\begin{proof} We will show that there exists a $K>0$ such that \begin{align}
P_\Z\left[\exp\left(\theta\left(\log \Z_{\sqrt{2}}(T)-P_\Z\left[\log \Z_{\sqrt{2}}(T)\right]\right)\right)\right]\leq \exp\left(\frac{T\theta^2K}{1-|\theta|}\right).\label{thetaz}
\end{align}
for $|\theta|\in(0,1)$.

 For fixed $T\in\mathbb{N}$, we define $\sigma$-field\begin{align*}
&\mathcal{F}_i(T)=\sigma\left[\W(t,x):0\leq t\leq i,x\in\mathbb{R}\right]\\
&\tilde{\mathcal{F}}_i(T)=\sigma\left[\W(t,x): t\not\in \left[i-1,i\right],x\in\mathbb{R}\right].
\end{align*}
Then, we write \begin{align*}
\log \Z_{\sqrt{2}}(T)-P_\Z\left[\log \Z_{\sqrt{2}}(T)\right]=\sum_{i=1}^{n_T}V_{i}^T,
\end{align*}
where \begin{align*}
V_i^T=P_\Z\left[\left.\log \Z_{\sqrt{2}}(T)\right|\mathcal{F}_i(T)\right]-P_\Z\left[\left.\log \Z_{\sqrt{2}}(T)\right|\mathcal{F}_{i-1}(T)\right]
\end{align*}
are martingale differences. 
Here, we introduce new random variables \begin{align*}
\hat{\Z}_{\sqrt{2}}(i,T)&=P_\Z\left[\left.\Z_{\sqrt{2}}(T)\right|\hat{\mathcal{F}}_i\right]\\
&=\int_{\mathbb{R}^2}\Z_{\sqrt{2}}\left(i-1,x\right)\rho_{1}(x,y)\Z_{\sqrt{2}}^{(i,y)}\left(T\right)dxdy.
\end{align*}
Since it is clear that \begin{align*}
P_\Z\left[\left.\log \hat\Z_{\sqrt{2}}(i,T)\right|\mathcal{F}_{i-1}(T)\right]=P_\Z\left[\left.\log \hat{\Z}_{\sqrt{2}}(i,T)\right|\mathcal{F}_{i}(T)\right],
\end{align*}
we have \begin{align*}
V_i^T=P_\Z\left[\left.\log \frac{\Z_{\sqrt{2}}(T)}{\hat{\Z}_{\sqrt{2}}(i,T)}\right|\mathcal{F}_{i}(T)\right]-P_\Z\left[\left.\log \frac{\Z_{\sqrt{2}}(T)}{\hat{\Z}_{\sqrt{2}}(i,T)}\right|\mathcal{F}_{i-1}(T)\right]
\end{align*}

Also, we consider a new probability measure on $\mathbb{R}^2$ by \begin{align*}
&\mu^{(i)}_T\left(x,y\right)dxdy\\
&=\frac{1}{\hat{\Z}_{\sqrt{2}}(i,T)}\Z_{\sqrt{2}}\left(i-1,x\right)\rho_{1}(x,y)\Z_{\sqrt{2}}^{(i,y)}\left(T\right)dxdy.
\end{align*}
Then, it is clear that \begin{align*}
\frac{\Z_{\sqrt{2}}(T)}{\hat{\Z}_{\sqrt{2}}(i,T)}=\int_{\mathbb{R}^2}\frac{\Z_{\sqrt{2}}\left(i-1,x;i,y\right)}{\rho_{1}(x,y)} \mu_T^{(i)}(x,y)dxdy,
\end{align*}
and Jensen's inequality implies from Theorem \ref{cdpprop} (ii) and (iv) that \begin{align*}
0\leq -P_\Z\left[\left.\log \frac{\Z_{\sqrt{2}}(T)}{\hat{\Z}_{\sqrt{2}}(i,T)}\right|\mathcal{F}_{i-1}(T)\right]&\leq -P_\Z\left[\log \frac{\Z_{\sqrt{2}}\left(0,0;1,0\right)}{p_{1}(0)}\right]\\
&\leq C_9,
\end{align*}
where we have used that \begin{align*}
-P_\Z\left[\log \frac{\Z_{\sqrt{2}}\left(0,0;1,0\right)}{p_{t}(0)}\right]\leq C_9
\end{align*}
(see Corollary \ref{cor}).

Thus, we have from Jensen's inequality that \begin{align*}
P_\Z\left[\left.\exp\left(V_{i}(T)\right)\right|\mathcal{F}_{i-1}(T)\right]\leq e^C P_\Z\left[\left.P_\Z\left[\left.\frac{\Z_{\sqrt{2}}(T)}{\hat{\Z}_{\sqrt{2}}{(i,T)}}\right|\tilde{\mathcal{F}}_{i}(T)\right]\right|\mathcal{F}_{i-1}(T)\right]=e^C.
\end{align*}
Also, Jensen's inequality implies that  \begin{align*}
&P_\Z\left[\left.\exp\left(-V_{i}(T)\right)\right|\mathcal{F}_{i-1}(T)\right]\leq P_\Z\left[\left.P_\Z\left[\left.\frac{\hat{\Z}_{\sqrt{2}}(i,T)}{{\Z}_{\sqrt{2}}{(T)}}\right|\tilde{\mathcal{F}}_{i}(T)\right]\right|\mathcal{F}_{i-1}(T)\right]\\
&\leq P_\Z\left[\left.P_\Z\left[\left.\int_{\mathbb{R}^2}\left(\frac{\Z_{\sqrt{2}}\left(i-1,x;i,y\right)}{\rho_{1}(x,y)}\right)^{-1}\mu^{(i)}(x,y)dxdy\right|\tilde{\mathcal{F}}_{i}(T)\right]\right|\mathcal{F}_{i-1}(T)\right]\leq C_{10},
\end{align*}
where we have used that \begin{align*}
P_\Z\left[\left(\frac{\Z_{\sqrt{2}}\left(0,x;1,y\right)}{\rho_{1}(x,y)}\right)^{-1}\right]\leq C_{10}.
\end{align*}

Thus, we have confirmed conditions in \cite[Theorem 2.1]{LiuWat} so that we have proved \ref{thetaz}.

\end{proof}

We can find that the above proof is true when we replace $\Z_{\sqrt{2}}(T)$ by $\Z_{\sqrt{2}}(T,0)$. Therefore, we have the following corollary from (\ref{thetaz}). 
\begin{cor}\label{L1con}
We have
\begin{align*}
&\lim_{T\to \infty}\frac{1}{T}P_\Z\left[\left|\log \Z_{\sqrt{2}}(T)-P_\Z\left[\Z_{\sqrt{2}}(T)\right]\right|\right]=0
\intertext{and}
&\lim_{T\to \infty}\frac{1}{T}P_\Z\left[\left|\log \Z_{\sqrt{2}}(T,0)-P_\Z\left[\Z_{\sqrt{2}}(T,0)\right]\right|\right]=0.
\end{align*}
In particular, we have  \begin{align*}
\lim_{T\to \infty}\frac{1}{T}P_\Z\left[\Z_{\sqrt{2}}(T,0)\right]=-\frac{1}{6}.
\end{align*}
\end{cor}

\begin{proof}[Proof of Lemma \ref{ttheta}] 
The proof is similar to the proofs of Lemma \ref{maxsup} and Lemma \ref{partition1}.

Also, we will often use the equations in Appendix to compute integrals of functions of heat kernels.

We write \begin{align*}
&\Z_{\sqrt{2}}^x(T)=\int_{\mathbb{R}}\Z_{\sqrt{2}}^x(1,w)\Z_{\sqrt{2}}^{(1,w)}(T)dw\\
&=\int_{A(T)}\Z_{\sqrt{2}}^x(1,w)\Z_{\sqrt{2}}^{(1,w)}(T)dw\\
&+\int_{A(T)^c}\Z_{\sqrt{2}}^x(1,w)\Z_{\sqrt{2}}^{(1,w)}(T)dw\\
&=:I_1(T,x)+I_2(T,x),
\end{align*}
where $A(T)=[-a(T),a(T)]$ is a segment with length of order $T^3$. Hereafter, we will look at $I_1(T,x)$ and $I_2(T,x)$.

We will show in the lemmas below that \begin{align*}
&\varlimsup_{\theta\to 0}\varlimsup_{T\to \infty}\frac{1}{\theta T}\log P_\Z\left[\sup_{x\in [-1,1]}I_1(T,x)^\theta\right]\leq \lim_{T\to \infty}\frac{1}{T}P_\Z\left[\log \Z_{\sqrt{2}}(T)\right]
\intertext{and}
&\varlimsup_{T\to \infty}\frac{1}{ T}\log P_\Z\left[\sup_{x\in [-1,1]}I_2(T,x)^\theta\right]= -\infty.
\end{align*}
Thus, we complete the proof.
\end{proof}

\begin{lem}\label{zc}
We have that 
\begin{align*}
\varlimsup_{\theta\to 0}\varlimsup_{T\to \infty}\frac{1}{\theta T}\log P_\Z\left[\sup_{x\in [-1,1]}I_1(T,x)^\theta\right]\leq \lim_{T\to \infty}\frac{1}{T}P_\Z\left[\log \Z_{\sqrt{2}}(T)\right]
\end{align*}
\end{lem}

\begin{lem}\label{zt}
We have that for any $\theta \in (0,1)$\begin{align*}
\varlimsup_{T\to \infty}\frac{1}{ T}\log P_\Z\left[\sup_{x\in [-1,1]}I_2(T,x)^\theta\right]= -\infty.
\end{align*}
\end{lem}

\begin{proof}[Proof of Lemma \ref{zt}]
It is easy to see  from Lemma \ref{cdpprop} (i) that \begin{align*}
\varlimsup_{T\to \infty}\frac{1}{T}\log P_\Z\left[ I_2(T,0)^\theta\right]\leq \varlimsup_{T\to \infty}\frac{\theta}{T}\log \int_{A(T)^c}\rho_1(0,w)dw=-\infty.
\end{align*}

Thus, it is enough to show that \begin{align*}
\varlimsup_{T\to \infty}\frac{1}{T}\log P_\Z\left[ \sup_{x,y\in [-1,1]}\left|I_2(T,x)^\theta-I_2(T,y)^\theta\right|\right]=-\infty.
\end{align*}

Applying (\ref{Garpol}) to the continuous function $I_2(T,y)^\theta$ with $d=1$, $x=0$, \begin{align*}
&P_\Z\left[\sup_{y\in[-1,1]}\left|I_2(T,y)^\theta-I_2(T,0)^\theta\right|\right]\\
&\leq C_{p,q}\left(\int_{-1}^1\int_{-1}^1\frac{P_\Z\left[\left|I_2(T,s)-I_2(T,t)\right|^{\theta p}\right]}{|t-s|^{pq}}dsdt\right)^{\frac{1}{p}}.
\end{align*}
for some $p>1$, $q>0$ with $pq>2$.

Thus, we will show that for $\theta \in (0,1)$, there exist $p\geq 1$ and  $q>0$ with $pq>2$ such that 
\begin{align}
\varlimsup_{T\to \infty}\frac{1}{T}\log \left(\int_{-1}^1\int_{-1}^1\frac{P_\Z\left[\left|I_2(T,s)-I_2(T,t)\right|^{\theta p}\right]}{|t-s|^{pq}}dsdt\right)^{\frac{1}{p}}=-\infty.\label{i2t}
\end{align}

We remark that $I_2(T,x)$ have the following Wiener chaos representation:
\begin{align*}
I_2(T,x)&=\int_{A(T)^c}\rho_1(w-x)dw\\
&+\sum_{k\geq 1}2^{\frac{k}{2}}\int_{\Delta_{k}(T)}\int_{\mathbb{R}^k}\rho^{(k,T)}(x;{\bf t},{\bf x})\W(dt_1,dx_1)\cdots \W(dt_k,dx_k)\\
&=\sum_{k\geq 0} 2^{\frac{k}{2}}J^{(k)}(T,x),
\end{align*}
where \begin{align*}
&\rho^{(1,T)}(x;t,x_1)\\
&=\begin{cases}
\dis \dis \int_{A(T)^c}\rho_1(x,w)\rho_{t-1}(w,x_1)dw,\ \ \ \ &\text{for }1\leq t\leq T\\
\dis \rho_{t}(x,x_1)\int_{A(T)^c}\rho_{1-t}(x_{1},w)dw,\ \ \  \ &\text{for }0<t\leq  1,
\end{cases}
\end{align*}
and
\begin{align*}
&\rho^{(k,T)}(x;{\bf t},{\bf x})\\
&=\begin{cases}
\dis \int_{A(T)^c}\rho_1(x,w)\rho_{t_1-1}(w,x_1)\prod_{i=2}^{k}\rho_{t_i-t_{i-1}}(x_{i-1},x_i)dw,\ \ \\
\hspace{17em}\text{for }1\leq t_1<\cdots<t_k\leq T\\
\dis \rho_{t_1}(x,x_1)\prod_{\begin{smallmatrix}i=2,\\i\not=\ell+1\end{smallmatrix}}^{k}\rho_{t_i-t_{i-1}}(x_{i-1},x_i)\int_{A(T)^c}\rho_{1-t_{\ell}}(x_\ell,w)\rho_{t_{\ell+1}-1}(w,x_{\ell+1})dw,\ \ \\
\hspace{12em} \text{for }0<t_1<\cdots<t_\ell\leq  1<t_{\ell+1}<\cdots<t_k\leq T.
\end{cases}
\end{align*}

We will estimate \begin{align*}
P_\Z\left[|J^{(k)}(T,x)-J^{(k)}(T,y)|^2\right]
\end{align*}
for $k\geq 0$.
It is easy to see that \begin{align*}
|J^{(0)}(T,x)-J^{(0)}(T,y)|&\leq \int_{a(T)}^\infty \left|\rho_1(x-w)-\rho_1(y-w)\right|dw\\
&+\int^{-a(T)}_{-\infty} \left|\rho_1(x-w)-\rho_1(y-w)\right|dw\\
&\leq |x-y|\int_{a(T)-1}^\infty \frac{4}{\sqrt{2\pi}}w\exp\left(-\frac{w^2}{2}\right)dw\\
&=\frac{4}{\sqrt{2\pi}}|x-y|\exp\left(-\frac{(a(T)-1)^2}{2}\right).
\end{align*}
Also, we have \begin{align}
&P_\Z\left[|J^{(1)}(T,x)-J^{(1)}(T,y)|^2\right]\notag\\
&=\int_0^T\int_{\mathbb{R}}\left(\rho^{(1,T)}(x;t,x_1)-\rho^{(1,T)}(y;t,x_1)\right)^2dtdx_1\notag\\
&\stackrel{(\ref{hconv})}{=}\int_1^Tdt\iint_{(A(T)^c)^2}(\rho_1(x,w)-\rho_1(y,w))(\rho_1(x,w')-\rho_1(y,w'))\rho_{2(t-1)}(w,w')dwdw'\notag\\
&+\int_0^1dt\int_\mathbb{R}dx_1(\rho_t(x,x_1)-\rho_t(y,x_1))^2\iint_{(A(T)^c)^2}\rho_{1-t}(x_1,w)\rho_{1-t}(x_1,w')dwdw'\notag\\
&=: M^{(1)}(T)+M^{(2)}(T)\notag\\
&\stackrel{\textrm{H\"older}}{\leq} \int_1^T\left(\int_{A(T)^c}dw(\rho_1(x,w)-\rho_1(y,w))^2\int_{\mathbb{R}}dw'\rho_{2(t-1)}(w,w')^2\right)\notag\\
&+\int_0^1dt\int_\mathbb{R}dx_1\int_{A(T)^c}dw(\rho_t(x,x_1)-\rho_t(y,x_1))^2\rho_{1-t}(x_1,w)^2\notag\\
&\stackrel{\text{H\"older}, (\ref{hprod}), (\ref{hconv})}{\leq} \int_{1}^{T}dt\frac{1}{2\sqrt{2\pi(t-1) }}\left(\int_{\mathbb{R}}dw(\rho_1(x,w)-\rho_1(y,w))^2\right)^{1/2}\notag\\
&\hspace{6em}\times \left(\int_{A(T)^c}2\rho_1(x,w)^2dw+\int_{A(T)^c}2\rho_1(y,w)^2dw\right)^{1/2}\notag\\
&+\int_0^1\frac{dt}{2\sqrt{\pi (1-t)}}\int_{A(T)^c}dw\left(\rho_{2t}(0)\rho_{\frac{1}{2}}(x,w)+\rho_{2t}(0)\rho_{\frac{1}{2}}(y,w)-2\rho_{2t}(x,y)\rho_{\frac{1}{2}}\left(\frac{x+y}{2},w\right)\right)\notag\\
&\hspace{-1em}\stackrel{(\ref{hsqare}),(\ref{hconv})}{\leq} C|x-y|\exp\left(-C'a^2(T)\right)\notag\\
&+\int_0^1\frac{dt}{2\sqrt{\pi (1-t)}}\int_{A(T)^c}dw(\rho_{2t}(0)-\rho_{2t}(x,y))(\rho_{\frac{1}{2}}(x,w)+\rho_{\frac{1}{2}}(y,w))\notag\\
&+\int_0^1\frac{dt}{2\sqrt{\pi (1-t)}}\int_{A(T)^c}dw\rho_{2t}(x,y)\left(\rho_{\frac{1}{2}}(x,w)+\rho_{\frac{1}{2}}(y,w)-2\rho_{\frac{1}{2}}\left(\frac{x+y}{2},w\right)\right)\notag\\
&\stackrel{(\ref{hest})}{\leq} C|x-y|\exp\left(-C'a^2(T)\right).\notag
\end{align}

Also, we have
\begin{align*}
&P_\Z\left[|J^{(2)}(T,x)-J^{(2)}(T,y)|^2\right]=\int_{\Delta_2(T)}\int_{\mathbb{R}^2}\left(\rho_1^{(2,T)}(x;{\bf t},{\bf x})-\rho_1^{(2,T)}(y;{\bf t},{\bf x})\right)^2d{\bf t}d{\bf x}\\
&= \int_{D_{2}(1,T)}\int_{\mathbb{R}^2}\left(\int_{A(T)^c}(\rho_1(x,w)-\rho_1(y,w))\rho_{t_1-1}(w,x_1)dw\right)^2\rho_{t_2-t_{1}}(x_{1},x_{2})^2d{\bf t}d{\bf x}\\
&+\int_{0<t_1<t_2\leq  1}\int_{\mathbb{R}^2}(\rho_{t_1}(x,x_1)-\rho_{t_1}(y,x_1))^2\rho_{t_2-t_{1}}(x_{1},x_{2})^2	\left(\int_{A(T)^c}\rho_{1-t_{2}}(x_{2},w)dw\right)^2d{\bf t}d{\bf x}\\
&+\int_{0<t_1\leq   1<t_{2}\leq T}\int_{\mathbb{R}^2}(\rho_{t_1}(x,x_1)-\rho_{t_1}(y,x_1))^2 \left(\int_{A(T)^c}\rho_{1-t_{1}}(x_{1},w)\rho_{t_{2}-1}(w,x_{2})dw\right)^2d{\bf t}d{\bf x}\\
&\leq \int_{1}^Tdt_1\frac{\sqrt{T-t_1}}{\sqrt{\pi}}\int_{A(T)^c}\int_{A(T)^c}\left(\rho_1(x,w)-\rho_1(y,w)\right)(\rho_1(x,w')-\rho_1(y,w'))\rho_{2(t_1-1)}(w,w')dwdw'\\
&+\int_{0<t_1<t_2\leq 1}\int_{\mathbb{R}}(\rho_{t_1}(x,x_1)-\rho_{t_1}(y,x_1))^2\frac{2\sqrt{\pi (1-t_1)}}{2\sqrt{\pi(t_2-t_1)}2\sqrt{\pi(1-t_2)}}\left(\int_{A(T)^c}\rho_{1-t_1}(x_1,w)dw\right)^2d{\bf t}dx_1\\
&+\int_{0<t_1\leq   1<t_{2}\leq T}\int_{\mathbb{R}^2}(\rho_{t_1}(x,x_1)-\rho_{t_1}(y,x_1))^2 \\
&\hspace{4em}\times \left(\iint_{(A(T)^c)^2}\rho_{1-t_{1}}(x_{1},w)\rho_{1-t_1}(x_1,w')\rho_{2(t_2-1)}(w,w')dwdw'\right)d{\bf t}d{\bf x}\\
&\leq \frac{\sqrt{T-1}}{\sqrt{\pi}}M^{(1)}(T)+\frac{\sqrt{\pi}}{2}M^{(2)}(T)+\frac{\sqrt{T-1}}{\sqrt{\pi}}M^{(2)}(T).
\end{align*}

For $k\geq 3$,\begin{align*}
&P_\Z\left[|J^{(k)}(T,x)-J^{(k)}(T,y)|^2\right]=\int_{\Delta_k(T)}\int_{\mathbb{R}^k}\left(\rho_1^{(k,T)}(x;{\bf t},{\bf x})-\rho_1^{(k,T)}(y;{\bf t},{\bf x})\right)^2d{\bf t}d{\bf x}\\
&= \int_{D_{k}(1,T)}\int_{\mathbb{R}^k}\left(\int_{A(T)^c}(\rho_1(x,w)-\rho_1(y,w))\rho_{t_1-1}(w,x_1)dw\right)^2\prod_{i=2}^{k}\rho_{t_i-t_{i-1}}(x_{i-1},x_{i})^2d{\bf t}d{\bf x}\\
&+\int_{0<t_1\leq  1<t_{2}<\cdots<t_k\leq T}\int_{\mathbb{R}^k}(\rho_{t_1}(x,x_1)-\rho_{t_1}(y,x_1))^2\prod_{i=3}^{k}\rho_{t_i-t_{i-1}}(x_{i-1},x_{i})^2	\\
&\hspace{10em}\times \left(\int_{A(T)^c}\rho_{1-t_{1}}(x_{1},w)\rho_{t_{2}-1}(w,x_{2})dw\right)^2d{\bf t}d{\bf x}\\
&+\sum_{\ell=2}^{k-1}\int_{0<t_1<\cdots<t_\ell\leq  1<t_{\ell+1}<\cdots<t_k\leq T}\int_{\mathbb{R}^k}(\rho_{t_1}(x,x_1)-\rho_{t_1}(y,x_1))^2\prod_{\begin{smallmatrix}i=2,\\i\not=\ell+1\end{smallmatrix}}^{k}\rho_{t_i-t_{i-1}}(x_{i-1},x_{i})^2	 	\\
&\hspace{10em}\times \left(\int_{A(T)^c}\rho_{1-t_{\ell}}(x_{\ell},w)\rho_{t_{\ell+1}-1}(w,x_{\ell+1})dw\right)^2d{\bf t}d{\bf x}\\
&\stackrel{(\ref{hconv}),(\ref{hfull}),(\ref{hfull2})}{=}\frac{1}{2^{k-1}\Gamma(\frac{k+1}{2})}\int^{T}_1{(T-t_1)^{\frac{k-1}{2}}}\\
&\hspace{5em}\times \iint_{(A(T)^c)^2}(\rho_1(x,w)-\rho_1(y,w))(\rho_1(x,w')-\rho_1(y,w'))\rho_{2(t_1-1)}(w,w')dwdw'dt_1\\
&+\int_0^1dt_1\int_\mathbb{R}dx_1(\rho_{t_1}(x,x_1)-\rho_{t_1}(y,x_1))^2\\
&\hspace{5em}\times \int_1^Tdt_2\frac{(T-t_2)^{\frac{k-2}{2}}}{2^{k-2}\Gamma\left(\frac{k}{2}\right)}\left(\int_{A(T)^c}\int_{A(T)^c}\rho_{1-t_{1}}(x_{1},w)\rho_{1-t_1}(x_1,w')\rho_{2t_{2}-2}(w,w')dwdw'\right)\\
&+\sum_{\ell=2}^{k-1}\frac{1}{2^{k-2}\Gamma\left(\frac{k-\ell+1}{2}\right)\Gamma\left(\frac{\ell-1}{2}\right)}\int_{0}^1dt_1\int_{t_1}^1dt_\ell (t_\ell-t_1)^{\frac{\ell-3}{2}}\int_{\mathbb{R}}dx_1(\rho_{t_1}(x,x_1)-\rho_{t_1}(y,x_1))^2\rho_{\frac{t_\ell-t_1}{2}}(x_1,x_\ell)\\
&\hspace{4em}\times \int_1^Tdt_{\ell+1}\iint_{(A(T)^c)^2}{(T-t_{\ell+1})^{\frac{k-\ell-1}{2}}} \rho_{1-t_\ell}(w-x_\ell)\rho_{1-t_\ell}(w'-x_\ell)\rho_{2t_{\ell+1}-2}(w,w')dwdw'\\
&\leq \frac{(T-1)^{\frac{k-1}{2}}}{2^{k-1}\Gamma\left(\frac{k+1}{2}\right)}M^{(1)}(T)+\frac{(T-1)^{\frac{k-1}{2}}}{2^{k-1}\Gamma\left(\frac{k+1}{2}\right)}M^{(2)}(T)\\
&+\sum_{\ell=2}^{k-1}\frac{(T-1)^{\frac{k-\ell}{2}}}{2^{k-1}\Gamma\left(\frac{k-\ell+2}{2}\right)}\int_0^1dt_1\int_{t_1}^1dt_\ell \frac{(t_\ell-t_1)^{\frac{\ell-3}{2}}}{\Gamma\left(\frac{\ell-1}{2}\right)}\\
&\hspace{4em}\times\iint_{\mathbb{R}^2}dx_1dx_\ell(\rho_{t_1}(x,x_1)-\rho_{t_1}(y,x_1))^2\rho_{\frac{t_\ell-t_1}{2}}(x_1,x_\ell)\left(\int_{A(T)^c}\rho_{1-t_{\ell}}(x_{\ell},w)dw\right)^2\\
&\leq \frac{(T-1)^{\frac{k-1}{2}}}{2^{k-1}\Gamma\left(\frac{k+1}{2}\right)}M^{(1)}(T)+\frac{(T-1)^{\frac{k-1}{2}}}{2^{k-1}\Gamma\left(\frac{k+1}{2}\right)}M^{(2)}(T)\\
&+\sum_{\ell=2}^{k-1}\frac{(T-1)^{\frac{k-\ell}{2}}}{2^{k-1}\Gamma\left(\frac{k-\ell+2}{2}\right)}\int_0^1dt_1(1-t_1)^{\frac{1}{2}}\int_{t_1}^1dt_\ell \frac{(t_\ell-t_1)^{\frac{\ell-3}{2}}}{\sqrt{1-t_\ell}\Gamma\left(\frac{\ell-1}{2}\right)}\\
&\hspace{5em}\times \int_{\mathbb{R}}dx_1(\rho_{t_1}(x,x_1)-\rho_{t_1}(y,x_1))^2\left(\int_{A(T)^c}\rho_{1-t_{1}}(x_{1},w)dw\right)^2\\
&\leq \frac{(T-1)^{\frac{k-1}{2}}}{2^{k-1}\Gamma\left(\frac{k+1}{2}\right)}M^{(1)}(T)+\frac{(T-1)^{\frac{k-1}{2}}}{2^{k-1}\Gamma\left(\frac{k+1}{2}\right)}M^{(2)}(T)\\
&+\sum_{\ell=2}^{k-1}\frac{\sqrt{\pi}(T-1)^{\frac{k-\ell}{2}}}{2^{k-1}\Gamma\left(\frac{k-\ell+2}{2}\right)\Gamma\left(\frac{\ell}{2}\right)}M^{(2)}(T)\\
&\leq C\frac{T^{\frac{k-1}{2}}}{\Gamma\left(\frac{k-2}{2}\right)}(M^{(1)}(T)+M^{(2)}(T)).
\end{align*}

Thus, we have that \begin{align*}
P_\Z\left[|J^{(k)}(T,x)-J^{(k)}(T,y)|^2\right]\leq \frac{CT^{\frac{k-1}{2}}}{2^k\Gamma(\frac{k-2}{2})}P_\Z\left[|J^{(1)}(T,x)-J^{(1)}(T,y)|^2\right],
\end{align*}
where $C$ is a constant independent of $k$.

By hypercontractivity of Wiener chaos \cite[Theorem 5.10]{Jan}, we have that for $p\geq 2$\begin{align*}
P_\Z\left[|I_2(T,x)-I_2(T,y)|^p\right]^{1/p}&\leq \sum_{k\geq 0}P_\Z\left[|J^{(k)}(T,x)-J^{(k)}(T,y)|^p\right]^{1/p}\\
&\leq \sum_{k\geq 0}(p-1)^{k/2}P_\Z\left[|J^{(k)}(T,x)-J^{(k)}(T,y)|^2\right]^{1/2}\\
&\leq C_p|x-y|^{1/2}\exp\left(-C'a(T)\right).
\end{align*}

Thus,  (\ref{i2t}) holds with $p=\frac{10}{\theta}$ and $q=\frac{\theta}{4}$.

\end{proof}

\begin{proof}[Proof of Lemma \ref{zc}]
It is clear that 
\begin{align*}
\sup_{x\in [-1,1]}|I_1(T,x)|^\theta&\leq \sup_{x\in [-1,1],w\in A(T)}\left|\frac{\Z_{\sqrt{2}}^x(1,w)}{\rho_L(w-1)+\rho_L(w+1)}\right|^\theta\\
&\hspace{2em}\times \left( \int_{\mathbb{R}}\left(\rho_L(u-1)+\rho_L(u+1)\right)\Z_{\sqrt{2}}^{(1,u)}(T)du\right)^\theta,
\end{align*}
where $L\in \mathbb{N}$ is taken large later. Thus, we have that \begin{align*}
P_\Z\left[\sup_{x\in[-1,1]}|I_1(T,x)|^\theta\right]&\leq P_\Z\left[\sup_{x\in [-1,1],w\in A(T)}\left|\frac{\Z_{\sqrt{2}}^x(1,w)}{\rho_L(w-1)+\rho_L(w+1)}\right|^\theta\right]\\
&\times P_\Z\left[ \left( \int_{\mathbb{R}}\left(\rho_L(u-1)+\rho_L(u+1)\right)\Z_{\sqrt{2}}^{(1,u)}(T)du\right)^\theta\right].
\end{align*}

If there exists a constant $C>0$ such that  \begin{align}
P_\Z\left[\sup_{x\in [-1,1],w\in [2k-1,2k+1]}\left|\frac{\Z_{\sqrt{2}}^x(1,w)}{\rho_L(w-1)+\rho_L(w+1)}\right|\right]\leq C\label{boundedat}
\end{align}
for $k\in \mathbb{Z}$, then we have \begin{align*}
&P_\Z\left[\sup_{x\in[-1,1]}|I_1(T,x)|^\theta\right]\\
&\leq C^\theta a(T)P_\Z\left[ \left( \int_{\mathbb{R}}\left(\rho_L(u-1)+\rho_L(u+1)\right)\Z_{\sqrt{2}}^{(1,u)}(T)du\right)^\theta\right]
\end{align*}
and therefore \begin{align*}
&\varlimsup_{\theta\to 0}\varlimsup_{T\to\infty}\frac{1}{T\theta}\log P_\Z\left[\sup_{x\in[-1,1]}|I_1(T,x)|^\theta\right]\\
&\leq \varlimsup_{\theta\to 0}\varlimsup_{T\to\infty}\frac{1}{T\theta}\log P_\Z\left[ \left( \int_{\mathbb{R}}\left(\rho_L(u-1)+\rho_L(u+1)\right)\Z_{\sqrt{2}}^{(1,u)}(T)du\right)^\theta\right].
\end{align*}

Also, we know that \begin{align*}
&P_\Z\left[ \left( \int_{\mathbb{R}}\rho_L(u-1)\Z_{\sqrt{2}}^{(1,u)}(T)du\right)^\theta\right]\\
&=P_\Z\left[ \left( \Z_{\sqrt{2}}^1(T+L)\frac{1}{\dis \int_{\mathbb{R}}\frac{\Z_{\sqrt{2}}^1(L,u)}{\rho_L(u-1)}\nu^{(1,L)}(u)du}\right)^\theta\right]\\
&\leq P_\Z\left[ \left( \Z_{\sqrt{2}}^1(T+L)\right)^{\frac{\theta}{1-\theta}}\right]^{1-\theta}P_\Z\left[\frac{1}{\dis \int_{\mathbb{R}}\frac{\Z_{\sqrt{2}}^1(L,u)}{\rho_L(u-1)}\nu^{(1,L)}(u)du}\right]^\theta\\
&\leq P_\Z\left[ \left( \Z_{\sqrt{2}}^1(T+L)\right)^{\frac{\theta}{1-\theta}}\right]^{1-\theta}P_\Z\left[\frac{{\rho_L(u-1)}}{{\Z_{\sqrt{2}}^1(L,u)}}\right]^\theta\\
&\leq CP_\Z\left[ \left( \Z_{\sqrt{2}}^1(T+L)\right)^{\frac{\theta}{1-\theta}}\right]^{1-\theta},
\end{align*}
where $\nu^{(1,L)}(u)$ is the probability density function on $\mathbb{R}$ given by \begin{align*}
\nu^{(1,L)}(u)=\frac{1}{\dis \int_{\mathbb{R}}\rho_L(u-1)\Z_{\sqrt{2}}^{(1,u)}(T)du}\rho_L(u-1)\Z_{\sqrt{2}}^{(1,u)}(T).
\end{align*} 
Then, we have  from Lemma \ref{wttheta} that \begin{align*}
&\varlimsup_{\theta\to 0}\varlimsup_{T\to \infty}\frac{1}{T\theta}\log P_\Z\left[ \left( \int_{\mathbb{R}}\rho_L(u-1)\Z_{\sqrt{2}}^{(1,u)}(T)du\right)^\theta\right]\\
&\leq \mathcal{F}_\Z(\sqrt{2})
\end{align*}
and we can  complete the proof of Lemma \ref{ttheta}.

We will prove (\ref{boundedat}).

\vspace{2em}
We consider a function on $[-1,1]\times \mathbb{R}$
\begin{align*}
f(x,w)=\frac{\Z_{\sqrt{2}}^x(1,w)}{\rho_L(w-1)+\rho_L(w+1)}.
\end{align*}

Then, we have  from Lemma \ref{cdpprop} (i) that \begin{align*}
P_\Z\left[f(x,w)\right]= \frac{\rho_1(x,w)}{\rho_L(w-1)+\rho_L(w+1)}\leq C_L.
\end{align*}

Also, if  $w\geq w'\geq 1$, \begin{align*}
|f(x,w)-f(x',w')|\leq &\left|\frac{\Z_{\sqrt{2}}^x(1,w)-\Z_{\sqrt{2}}^{x'}(1,w')}{\rho_L(w-1)+\rho_L(w+1)}\right|\\
&+\Z_{\sqrt{2}}^{x'}(1,w')\left|\frac{1}{\rho_L(w-1)+\rho_L(w+1)}-\frac{1}{\rho_L(w'-1)+\rho_L(w'+1)}\right|\\
\leq & \left|\frac{\Z_{\sqrt{2}}^x(1,w)-\Z_{\sqrt{2}}^{x'}(1,w')}{\rho_L(w-1)+\rho_L(w+1)}\right|\\
&+\frac{|w-w'|^2}{L}\frac{\Z_{\sqrt{2}}^{x'}(1,w')\rho_L(w+1)}{\rho_L(w'-1)^2}.
\end{align*}
We can treat the case $w,w'\leq -1$ in the same manner and if $w,w'\in [-1,1]$, then it is clear that \begin{align*}
|f(x,w)-f(x',w')|&\leq C(L)\left(\left|{\Z_{\sqrt{2}}^x(1,w)-\Z_{\sqrt{2}}^{x'}(1,w')}\right|+|w-w'|^2\right).
\end{align*}

Thus, if we show for some $p\geq 1$ and $q>0$ with $pq>4$, there exists $\eta_p>pq-2$ such that \begin{align}
&P_\Z\left[\left|{\Z_{\sqrt{2}}^x(1,w)-\Z_{\sqrt{2}}^{x'}(w')}\right|^p\right]\\
&\leq C(L)\left(|x-x'|^{\eta_p}+|w-w'|^{\eta_p}\right)\left(\rho_L(|w|\vee |w'|-1)^p+\rho_L(|w|\vee |w'|+1)^p\right)\label{1timesup}
\intertext{and}
&P_\Z\left[\Z_{\sqrt{2}}^{x'}(1,w')^p\right]\leq C(L)\left(\rho_L(|w|\vee |w'|-1)^p+\rho_L(|w|\vee |w'|+1)^p\right),\label{1timesp}
\end{align}
then we can apply Lemma \ref{Gar} with $d=2$ to $f(x,w)$ and we obtain (\ref{boundedat}).

$\Z_{\sqrt{2}}^x(1,w)$ has the Wiener chaos representation \begin{align*}
\Z_{\sqrt{2}}^x(1,w)&=\rho_1(w-x)+\sum_{k\geq 1}\sqrt{2}^k\int_{D_k(1)} \rho^{(k)}(x,w;{\bf t},{\bf x})W(t_1,x_1)\cdots W(t_k,x_k)\\
&=\sum_{k\geq 0}\sqrt{2}^k{K^{(k)}(x,w)},
\end{align*}
where \begin{align*}
\rho^{(k)}(x,w;{\bf t},{\bf x})=\rho_{t_1}(x_1-x)\prod_{i=2}^k\rho_{t_i-t_{i-1}}(x_i-x_{i-1})\rho_{1-t_k}(w-x_k).
\end{align*}
Then, we have from (\ref{hfull}) that \begin{align*}
P_\Z\left[\left(K^{(k)}(x,w)\right)^2\right]=\frac{1}{2^{k+1}\Gamma\left(\frac{k+1}{2}\right)}\exp\left(-(x-w)^2\right),
\end{align*}
and hypercontractivity implies that \begin{align*}
&P_\Z\left[\Z_{\sqrt{2}}^{x'}(1,w')^p\right]\leq \left(\sum_{k\geq 0}(p-1)^{\frac{k}{2}}\left(\frac{1}{2^{k+1}\Gamma\left(\frac{k+1}{2}\right)}\exp\left(-(x'-w')^2\right)\right)^{1/2}\right)^p\\
&\leq C(p)\exp\left(-\frac{p(x'-w')^2}{2}\right).
\end{align*}

Then, we have that for $k\geq 2$\begin{align*}
&P_\Z\left[\left|K^{(k)}(x,y)-K^{(k)}(x',y')\right|^2\right]\\
&=\int_{D_k(1)}\int_{\mathbb{R}^k}\left(\rho_{{t_1}}(x,x_1)\rho_{1-t_k}(x_k,y)-\rho_{t_1}(x',x_1)\rho_{1-t_k}(x_k,y')\right)^2\prod_{i=1}^{k-1}\rho_{(t_{i+1}-t_i)}(x_{i},x_{i+1})^2d{\bf x}_kd{\bf t}_k\\
&=\frac{1}{2^{k-1}\Gamma\left(\frac{k-1}{2}\right)}\int_{0}^1\int_{s}^1(t-s)^{\frac{k-3}{2}}\left(\rho_{2s}(0)\rho_{2({1-t})}(0)\rho_{\frac{1}{2}}(x,y)+\rho_{2s}(0)\rho_{2({1-t})}(0)\rho_{\frac{1}{2}}(x',y')\right.\\
&\hspace{10em}\left.-2\rho_{2s}(x,x')\rho_{2(1-t)}(y,y')\rho_{\frac{1}{2}}\left(\frac{y+y'}{2}\frac{x+x'}{2}\right)\right)dtds\\
&=\frac{1}{2^{k+1}\Gamma\left(\frac{k+1}{2}\right)}\left(\rho_{\frac{1}{2}}(x,y)+\rho_{\frac{1}{2}}(x',y')-2\rho_{\frac{1}{2}}\left(\frac{x+x'}{2},\frac{y+y'}{2}\right)\right)\\
&+\frac{2}{2^{k-1}\Gamma\left(\frac{k-1}{2}\right)}\int_0^1\int_s^1(t-s)^{\frac{k-3}{2}}\rho_{\frac{1}{2}}\left(\frac{x+x'}{2},\frac{y+y'}{2}\right)\rho_{2s}(0)\left(\rho_{2({1-t})}(0)-\rho_{2(1-t)}(y,y')\right)dtds\\
&+\frac{2}{2^{k-1}\Gamma\left(\frac{k-1}{2}\right)}\int_0^1\int_s^1(t-s)^{\frac{k-3}{2}}\rho_{\frac{1}{2}}\left(\frac{x+x'}{2},\frac{y+y'}{2}\right)\rho_{2(1-t)}(y,y')\left(\rho_{2s}(0)-\rho_{2s}(x,x')\right)dtds\\
&\leq \frac{C(|x-x'|+|y-y'|)}{2^{k-1}\Gamma\left(\frac{k-1}{2}\right)}\exp\left(-\frac{(|y|\vee |y'|-1)^2 }{2}\right).
\end{align*}

Also, we can estimate that \begin{align*}
P_\Z\left[\left|K^{(0)}(x,y)-K^{(0)}(x',y')\right|^2\right]\leq C(|x-x'|^2+|y-y'|^2)\exp\left(-(|y|\vee |y'|-1)^2\right)
\end{align*}
and \begin{align*}
&P_\Z\left[\left|K^{(1)}(x,y)-K^{(1)}(x',y')\right|^2\right]\\
&\leq \int_0^1\left(\rho_{2s}(0)\rho_{2(-1s)}(0)\left(\rho_{\frac{1}{2}}(x,y)+\rho_{\frac{1}{2}}(x',y')\right)\right.\\
&\hspace{5em}
\left.-2\rho_{2s}(x,x')\rho_{2(1-s)}(y,y')\rho_{\frac{1}{2}}\rho\left(\frac{x+x'}{2},\frac{y+y'}{2}\right)\right)ds\\
&\leq {C(|x-x'|+|y-y'|)}\exp\left(-\frac{(|y|\vee |y'|-1)^2 }{2}\right).
\end{align*}

Then, hypercontractivity implies that \begin{align*}
&P_\Z\left[\left|\Z_{\sqrt{2}}^x(1,w)-\Z_{\sqrt{2}}^{x'}(1,w')\right|^p\right]\\
&\leq \left(|x-x'|^{\frac{p}{2}}+|w-w'|^{\frac{p}{2}}\right)\left(\sum_{k\geq 0}(p-1)^{\frac{k}{2}}\left(\frac{C}{2^{k-1}\Gamma\left(\frac{k-1}{2}\right)}\exp\left(-\frac{2(w^2+w'^2)}{L}\right)\right)^{1/2}\right)^p\\
&\leq C(p)(|x-x'|^{\frac{p}{2}}+|w-w'|^{\frac{p}{2}})\exp\left(-\frac{2p(w^2+w'^2)^2}{L}\right)
\end{align*}
for $L$ large enough.

Thus, we have confirmed (\ref{1timesup}) and (\ref{1timesp}). Therefore, we completed the proof of Lemma \ref{ttheta}.
\end{proof}

Finally, we need to prove the free energy $F_\Z(\sqrt{2})=-\dis \frac{1}{6}$. The proof is a modification of the proof of Lemma \ref{ttheta}.

\begin{proof}[Proof of Lemma \ref{freeenecdp}]
It is easy to see that for  $a'(T)\in [0,\infty)$\begin{align*}
P_\Z\left[\Z_{\sqrt{2}}(T)^\theta\right]&\leq \sum_{k=-a'(T)}^{a'(T)}P_\Z\left[\left(\int_{2k-1}^{2k+1}\Z_{\sqrt{2}}(T,x)dx\right)^\theta\right]\\
&+P_\Z\left[\int_{-\infty}^{-a'(T)}\Z_{\sqrt{2}}(T,x)dx\right]^\theta+P_\Z\left[\int^{\infty}_{a'(T)}\Z_{\sqrt{2}}(T,x)dx\right]^\theta.
\end{align*}

If $\lim_{T\to \infty}\frac{a'(T)}{T^3}>0$, then \begin{align*}
\varlimsup_{T\to\infty}\frac{1}{T}\log \left(P_\Z\left[\int_{-\infty}^{-a'(T)}\Z_{\sqrt{2}}(T,x)dx\right]^\theta+P_\Z\left[\int^{\infty}_{a'(T)}\Z_{\sqrt{2}}(T,x)dx\right]^\theta\right)=-\infty.
\end{align*}

We denote \begin{align*}
\exp\left(A_{\sqrt{2}}(T,x)\right)=\frac{\Z_{\sqrt{2}}(T,x)}{\rho_T(x)},\ \ T>0,\ \ x\in\mathbb{R}.
\end{align*}
Then, we find that \begin{align*}
P_\Z\left[\left(\int_{2k-1}^{2k+1}\Z_{\sqrt{2}}(T,x)dx\right)^\theta\right]&=P_\Z\left[\left(\int_{-1}^1\exp\left(A_{\sqrt{2}}(T,x)\right)\rho_T(x+2k)dx\right)^\theta\right]\\
&\leq P_\Z\left[\sup_{x\in[-1,1]}\exp\left(\theta A_{\sqrt{2}}(T,x)\right)\right]\int_{-1}^1\rho_{T}(x+2k)^\theta dx.
\end{align*}

Since \begin{align*}
\sum_{k=-\infty}^\infty \int_{-1}^1\rho_{T}(x+2k)^\theta dx<\infty,
\end{align*}
it is enough to show that \begin{align*}
\varlimsup_{\theta\to 0}\varlimsup_{T\to\infty}\frac{1}{T\theta}\log P_\Z\left[\sup_{x\in[-1,1]}\exp\left(\theta A_{\sqrt{2}}(T,x)\right)\right]\leq -\frac{1}{6}.
\end{align*}

When we consider the time reversal, it is enough to show that \begin{align*}
\varlimsup_{\theta\to 0}\varlimsup_{T\to\infty}\frac{1}{T\theta}\log P_\Z\left[\sup_{x\in[-1,1]}\Z^{x}_{\sqrt{2}}(T,x)^\theta\right]\leq -\frac{1}{6}.
\end{align*}
We know 
\begin{align*}
\Z^{x}_{\sqrt{2}}(T,x)&=\int_{A(T)}\Z^x_{\sqrt{2}}(1,w)\int_\mathbb{R}\Z_{\sqrt{2}}(1,w;T,0)dw\\
&+\int_{A(T)^c}\Z^x_{\sqrt{2}}(1,w)\int_\mathbb{R}\Z_{\sqrt{2}}(1,w;T,0)dw\\
&=I'_1(T,x)+I_2'(T,x)
\end{align*}
in a similar manner to the proof of Lemma \ref{ttheta}. Then, we find that \begin{align*}
P_\Z\left[\left(\Z^{x}_{\sqrt{2}}(T,x)\right)^\theta\right]\leq P_\Z\left[\sup_{x\in[-1,1]}|I_1'(T,x)|^\theta\right]+P_\Z\left[\sup_{x\in[-1,1]}|I_2'(T,x)|^\theta\right]
\end{align*}
and we obtain by using the same argument as the proof of Lemma \ref{ttheta} that \begin{align*}
\varlimsup_{\theta\to 0}\varlimsup_{T\to\infty}\frac{1}{T\theta}\log P_\Z\left[\sup_{x\in[-1,1]}|I_1'(T,x)|^\theta\right]&\leq \varlimsup_{T\to \infty}\frac{1}{T}P_\Z\left[\log \Z_{\sqrt{2}}(T,0)\right]=-\frac{1}{6}
\intertext{and}
\varlimsup_{\theta\to 0}\varlimsup_{T\to\infty}\frac{1}{T\theta}\log P_\Z\left[\sup_{x\in[-1,1]}|I_2'(T,x)|^\theta\right]&=-\infty.
\end{align*}
\end{proof}

\begin{appendix}
\section{Some formulas for heat kernel}
Here, we give some formulas for calculations in the proofs. We set \begin{align*}
\rho_t(x-y)=\rho_t(x,y)=\frac{1}{\sqrt{2\pi t}}\exp\left(-\frac{(y-x)^2}{2t}\right)
\end{align*} 
for $x,y\in \mathbb{R}$ and $t>0$. Then, we have that for $k\geq 1$
\begin{align}
&\rho_t(x)^2=\frac{1}{2\sqrt{\pi t}}\rho_{\frac{t}{2}}(x),\label{hsqare}\\
&\rho_{t}(x,w)\rho_t(y,w)=\rho_{2t}(x,y)\rho_{\frac{t}{2}}\left(\frac{x+y}{2},w\right),\label{hprod}\\
&\int_\mathbb{R}\rho_s(x,y)\rho_t(y,z)dy=\rho_{t+s}(x,z),\label{hconv}\\
&\int_{\mathbb{R}^k}\rho_{t_1-t_0}(x,x_1)^2\prod_{i=1}^{k-1}\rho_{(t_{i+1}-t_i)}{(x_{i},x_{i+1})}^2\rho_{t_{k+1}-t_k}(x_k,y)^2d{\bf x}_k\notag\\
&=\frac{1}{2^{k+1}\pi^{\frac{k+1}{2}}}\rho_{\frac{t_{k+1}-t_0}{2}}(x,y)\prod_{i=0}^{k}\frac{1}{\sqrt{t_{i+1}-t_i}},\label{hsqconv}\\
&\int_{t_0}^{t_2}\int_\mathbb{R}\rho_{t_1-t_0}(x_0,x_1)^2dx_1dt_1=\frac{\sqrt{t_2-t_0}}{\sqrt{\pi}}\label{h2}\\
&\int_{D_k(t_0,t_{k+1})}\int_{\mathbb{R}^k}\rho_{t_1-t_0}(x,x_1)^2\prod_{i=1}^{k-1}\rho_{(t_{i+1}-t_i)}(x_{i},x_{i+1})^2\rho_{t_{k+1}-t_k}(x_k,y)^2d{\bf x}_kd{\bf t}_k\notag\\
&=\frac{(t_{k+1}-t_0)^{\frac{k-1}{2}}}{2^{k+1}\Gamma\left(\frac{k+1}{2}\right)}\rho_{\frac{t_{k+1}-t_0}{2}}(x,y),\label{hfull}\\
&\int_{D_k(s,t)}\int_{\mathbb{R}^k}\rho_{t_1-t_0}(x,x_1)^2\prod_{i=1}^{k}\rho_{(t_{i+1}-t_i)}(x_{i},x_{i+1})^2d{\bf x}_kd{\bf t}_k\notag\\
&=\frac{(t-s)^{\frac{k}{2}}}{2^{k}\Gamma\left(\frac{k+2}{2}\right)},\label{hfull2}\\
&\int_0^t \left(\rho_s(0)-\rho_s(x)\right)ds\leq \frac{|x|}{2\sqrt{\pi}}\int_{\frac{|x|^2}{2t}}^\infty\frac{1}{u^{\frac{3}{2}}}\left(1-\exp\left(-u\right)\right)du\notag\\
&\hspace{5em}\leq \frac{|x|}{2\sqrt{\pi}}\int_0^\infty u^{-\frac{3}{2}}(1\wedge u)du,\label{hest}
\end{align}
where ${\bf x}_k=(x_1,\cdots,x_k)\in\mathbb{R}^k$, ${\bf t}_k=(t_1,\cdots,t_k)\in [0,\infty)^k$, and \begin{align*}
D_k(s,t)=\{{\bf t}_k\in [0,\infty)^k:s\leq t_1<\cdots<t_k<t\}. 
\end{align*}

\end{appendix}

\vspace{2em}

{\bf Acknowledgement:} 
This research was supported by JSPS Grant-in-Aid for Young Scientists (B) 26800051.



\end{document}